\documentclass[12pt,a4paper]{amsart}
\usepackage{amssymb, amstext, amscd, amsmath, color}
\usepackage{url}
\usepackage{hhline}

\usepackage{pb-diagram}
\usepackage[dvips]{graphicx}

\title[Infinite Frameworks and Operator Theory]{Infinite Bar-Joint Frameworks, Crystals and Operator Theory}
\author[John Owen]{J.C. Owen}
\address{D-Cubed, Siemens PLM Software, Park House,\\
Castle Park, Cambridge UK}
\email{owen.john.ext@siemens.com}

\author[Stephen Power]{S.C.  Power}
\address{Dept.\ Math.\ Stats.\\ Lancaster University\\
Lancaster LA1 4YF \\U.K. }
\email{s.power@lancaster.ac.uk}



\theoremstyle{plain}
\newtheorem{thm}{Theorem}[section]
\newtheorem{cor}[thm]{Corollary}
\newtheorem{prop}[thm]{Proposition}
\newtheorem{lem}[thm]{Lemma}
\theoremstyle{definition}
\newtheorem{rem}[thm]{Remark}

\newtheorem{definition}[thm]{Definition}
\newtheorem{example}[thm]{Example}

\newcommand{\bC}{{\mathbb{C}}}

\newcommand{\bN}{{\mathbb{N}}}

\newcommand{\bR}{{\mathbb{R}}}
\newcommand{\bT}{{\mathbb{T}}}

\newcommand{\bZ}{{\mathbb{Z}}}

  \newcommand{\B}{{\mathcal{B}}}
  \newcommand{\C}{{\mathcal{C}}}
  \newcommand{\D}{{\mathcal{D}}}
  \newcommand{\E}{{\mathcal{E}}}
  \newcommand{\F}{{\mathcal{F}}}
  \newcommand{\G}{{\mathcal{G}}}
\renewcommand{\H}{{\mathcal{H}}}

  \newcommand{\K}{{\mathcal{K}}}

  \newcommand{\M}{{\mathcal{M}}}
  \newcommand{\N}{{\mathcal{N}}}

\renewcommand{\P}{{\mathcal{P}}}
  
  \newcommand{\R}{{\mathcal{R}}}

  \newcommand{\T}{{\mathcal{T}}}
  \newcommand{\U}{{\mathcal{U}}}
  \newcommand{\V}{{\mathcal{V}}}

\newcommand{\ol}{\overline  }

\newcommand{\coker}{\operatorname{coker}}

\begin{document}




\begin{abstract}
A theory of flexibility and rigidity is developed for general infinite bar-joint frameworks
$(G,p)$.
Determinations of nondeformability through vanishing flexibility are obtained as well
as sufficient conditions for deformability.
Forms of infinitesimal flexibility are defined in terms of the operator theory of the
associated infinite rigidity matrix $R(G,p)$. The matricial symbol function of an abstract crystal framework is introduced, being the matrix-valued function on the $d$-torus representing  $R(G,p)$ as a Hilbert space operator.
The symbol function is related to infinitesimal flexibility, deformability and isostaticity.
Various generic abstract crystal frameworks which
are in Maxwellian equilibrium, such as certain $4$-regular planar frameworks, are proven to be square-summably infinitesimally rigid as well as smoothly deformable in infinitely many ways.
The symbol function of a three-dimensional crystal framework
determines the infinitesimal wave flexes in models for the
low energy vibrational modes (RUMs) in material crystals.
For crystal frameworks
with inversion symmetry it is shown that the RUMS
appear in surfaces, generalising a result of F. Wegner \cite{weg} for
tetrahedral crystals.
\end{abstract}
\maketitle
\tableofcontents


\section{Introduction}

Infinite bar-joint frameworks appear frequently as idealised models in the analysis of deformations and vibration modes of amorphous and crystalline materials. See
\cite{gid-et-al}, \cite{ham-dov-gid-hei-win}, \cite{chu-tho}, \cite{goodwin2008},
\cite{weg} and \cite{wyart} for example and the comments below.
Despite these connections there has been no extended mathematical analysis of such models. Notions of rigidity, flexibility, deformability, constrainedness,
independence and isostaticity, for example, are usually employed either
in the sense of their  usage for
a finite approximating framework or in a manner drawn from experience and
empirical fact in the light of the application at hand.
It seems that a deeper understanding of the models is of considerable interest in
its own right and that a mathematical development may prove useful in certain applications. In what follows we shall provide formal definitions of the terms above
in quite a wide variety  of forms  and
we examine some of their inter-relationships and manifestations.

Suppose that one starts with a flexible square bar-joint framework in two dimensions and
that this is then extended periodically to create an infinite periodic bar-joint
network.
Is the resulting assemblage, with inextendible bars,
continuously flexible in two dimensional space?
A moment's reflection reveals a proliferation of
flexibility, such as sheering motions with one half of the network fixed.

However such movement is dramatically infinite and a natural second question is whether for
periodic frameworks such as these there are flexes for which the total joint movement is finite. The less obvious negative answer in these cases offers some satisfaction
in that it is consistent with so-called
Maxwell counting in the case where the average number of degrees of freedom of the joints matches the average
number of constraints per joint. (See Theorem 5.2 and Corollary 5.3.)

On the other hand if, for the grid example, one rigidifies alternate squares by adding diagonal bars, as shown, then the resulting structure
of corner linked rigid squares remains properly flexible, although now uniquely so. In fact the unique flex has an affinely contracting character (See Definition 4.3) with alternating
rotation of the squares.

More generally the flexibility of polytope networks in two and three dimensions continues to be of interest in the modeling  of crystals and  amorphous materials,  especially with regard to their
low frequency vibrational modes. Such modes appear, for example, in higher order symmetry
phases of tetrahedral crystals  and are  referred to
as rigid unit modes (RUMs).  Indeed in the paper of Giddy et al
\cite{gid-et-al} the alternating flex of the squares framework above has been
associated with vibrational modes in perovskite.
See also Hammond et al \cite{ham-dov-gid-hei-win},
Wegner \cite{weg}, as well as
Goodwin et al \cite{goodwin2008} for a useful overview. At the same time, in the modeling of
 amorphous materials, such as glasses,
there is interest in understanding the critical
probabilities that guarantee flexibility and rigidity for classes of randomly  constructed frameworks. See, for example, Chubynsky and Thorpe  \cite{chu-tho} for the recent determination of such probabilities in simulation experiments.

\begin{center}
\begin{figure}[h]
\centering
\includegraphics[width=3cm]{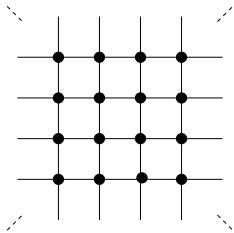}
\caption{The grid framework in the plane, $\G_{\bZ^2}$.}
\end{figure}
\end{center}

\begin{center}
\begin{figure}[h]
\centering
\includegraphics[width=3cm]{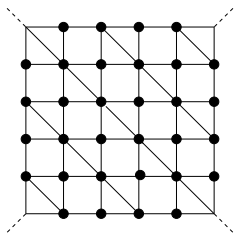}
\caption{The corner-joined squares framework, $\G_{sq}$.}
\end{figure}
\end{center}

Formally, a \textit{framework} in  ${\mathbb R}^d$ (or bar-joint framework,
or distance-constraint framework)
is a pair $\G =(G, p)$ where $G=(V,E)$
is a simple connected graph
and
$p=(p_1, p_2, \ldots )$ is a framework vector made up of framework points
$p_i$ in $\bR^d$ associated with the vertices $v_1, v_2, \dots $ of $V$.
The \textit{framework edges}  are the (closed) line segments $[p_i, p_j]$
associated with the edges $E$ of the graph $G=(V,E)$.  As the ellipsis suggest, we allow $G$ to be a countable graph.
We shall also define a \textit{crystal framework} $\C$   as a framework with translational symmetry which is generated by a connected finite {\it motif} of edges and vertices.
(See Definition 4.2.)

When $G$ is finite and the framework points are generically located in $\mathbb R^2$ then a celebrated theorem of Laman \cite{lam}, well-known in structural engineering and in the discrete
mathematics of rigidity matroids \cite{gra-ser-ser}, gives a simple combinatorial criterion for the minimal infinitesimal rigidity of the framework;  the graph itself satisfies Maxwell's counting rule $2|V| - |E| = 3$, and subgraphs $G'=(V', E')$ must comply with $2|V'| - |E'| \geq 3$.  This is a beautiful result since the rigidity
 here is the  noncombinatorial requirement that the kernel of an associated rigidity matrix $R(G,p)$
 has the smallest dimension (namely three) for some (and hence all) generic framework.
 On the other hand frameworks with global symmetries, or even with "symmetric elements" (such as parallel edges) are not generic, that is, algebraic dependencies do exist between the framework point coordinates. Such frameworks arise in  classical crystallography on the one hand  and in mathematical models in structural engineering and in materials science on the other.
See, for example,  Donev and Torquato \cite{don-tor}, Hutchinson and Fleck \cite{hut-fle}, Guest and Hutchinson \cite{gue-hut}
 and various papers in the  conference proceedings \cite{tho-dux}.

 The present paper develops two  themes. The first concerns a mathematical theory of deformability and rigidity for {general} infinite frameworks, with frequent attention to the case of periodic frameworks.
 There is, unsurprisingly, a great diversity of infinite framework
flexing phenomena and we introduce  strict terminology and some methods from functional analysis to capture some of this.
 In the second theme we propose an operator theory perspective for the infinitesimal
 (first order) flexibility of infinite frameworks.

Particularly interesting classes of infinite frameworks, from the point of view of flexibility, are those in the plane whose graphs  are $4$-regular and those in three dimensions whose graphs are $6$-regular. In this case the  graphs are in  Maxwell equilibrium, so to speak, and so in a generic framework realisation
any flex must activate countably many vertices.
This is also the case for
various periodic realisations such as the  kagome framework, $\G_{kag}$, formed by  corner-joined triangles in regular hexagonal arrangement, and  frameworks in three dimensions formed by pairwise corner-joined tetrahedra.
Despite being internally rigid in this way (Definition 2.18 (vi))
these frameworks admit diverse deformations. For example we note that the kagome framework
admits uncountably many distinct deformations and in Theorem 4.4 we note that
$\bZ^d$-periodic cell-generic grid frameworks in $\bR^d$ admit deformations
associated with affine transformations.

A significant phenomenon in the infinite setting is the appearance of
\textit{vanishing flexibility}. This means, roughly speaking, that the framework
 is a union of finite flexible subframeworks  but
the extent of flexibility diminishes to zero as the size of these subframeworks
increases, so that the infinite assemblage is inflexible. Elementary examples were indicated
in \cite{owe-pow-3} but we give more subtle examples here
which are due to flex amplification at second order distances through concatenation
effects. In particular there are  bounded infinitesimal flexes in periodic frameworks
that admit no continuous extensions and which do not arise as the derivative
of a smooth deformation.
We also note  that there are
$\bZ^2$-periodic crystal frameworks  which are somewhat paradoxical, being
indeformable despite the flexibility of all  supercell subframeworks.
On the other hand, in the positive direction, in Theorem \ref{t:aathm} we give a general result which identifies a uniform principle for the existence of a deformation. The proof uses the Ascoli-Arzela
theorem  on the precompactness
of equicontinuous families of local flexes. It remains an interesting open problem to determine necessary and sufficient conditions for the rigidity and bounded rigidity
of periodic planar frameworks.

The operator theory perspective for  frameworks was suggested  in  \cite{owe-pow-3}
as an approach to a wider understanding of  infinitesimal flexibility and rigidity.
In this consideration the rigidity matrix is infinite and determines  operators between various normed sequence spaces associated with nodes and with edges.
Also, in \cite{owe-pow-2} we have given a direct proof of the Fowler-Guest  formula \cite{fow-gue} for symmetric finite frameworks which is based on the commutation properties of the rigidity matrix as a linear transformation and this adapts readily to the infinite case and the rigidity operators of crystal frameworks.
Indeed, translational symmetry ensures that the rigidity
matrix $R(G,p)$  intertwines  the coordinate shift operations.
We consider square summable flexes and stresses and for distance regular
locally finite frameworks
$R(G,p)$ is interpreted as a bounded  operator between Hilbert spaces.
Also, enlarging to complex Hilbert spaces  the Fourier transform $\F R(G,p)\F^{-1}$ is identified as a multiplication operator
 \[
 M_\Phi : L^2(\bT^d) \otimes \bC^n \to  L^2(\bT^d) \otimes \bC^m
 \]
given by an $m \times n$ matrix-valued
function $\Phi(z)$ on the $d$-torus. The
function $\Phi$ for $\C$  is referred to   as the \textit{matricial symbol function}
associated with the particular generating motif. The  terminology  and notation is borrowed from standard usage for Toeplitz operators and multiplication operators (see \cite{bot-sil} for example). Many aspects of infinitesimal flexibility and isostaticity are expressible and analysable in terms of the matricial symbol function and its associated operator theory.
For example a straightforward consequence of the operator theoretic approach is the  square summable isostaticity of various nondegenerate regular frameworks that satisfy Maxwell counting, such as grid frameworks and the kagome framework.

An explicit  motif-to-matrix function algorithm  is given
for the progression $$\C =(G,p)\to R(G,p) \to \Phi(z).$$
Furthermore the identification
of infinitesimal periodic-modulo-phase flexes and their multiplicities
is determined by the degeneracies of $\Phi(z)$ as $z$ ranges on the $d$-torus.
In particular, the function
\[
\mu(z) := \dim \ker \Phi(z) : \bT^d \to \bZ_.
\]
gives a determination of the mode multiplicity of  periodic-modulo-phase
infinitesimal flexes.

In the discussions below we are mainly concerned with  properties of
mathematical bar-joint frameworks. (The framework bars are indestructibly inextensible, the joints are located deterministically, they maintain perfect, frictionless fit  and may even coincide.)  Nevertheless,  analysis of matricial symbol functions and their degeneracies are
particularly relevant to the description and analysis  of Rigid Unit Modes in material crystals.
We show that for crystal frameworks
with inversion symmetry  the set of RUMS
is a union of  surfaces. This  generalises and provides an alternative
perspective for a recent result of Wegner \cite{weg} for
tetrahedral crystals.

Operator theory methods have proven beneficial in many areas of mathematics
and applications. This is evidently the case for multi-linear systems theory  and in control theory for example. Infinite rigidity matrix analysis seems to possess some similitudes
with these areas and
it seems to us that here too the operator turn will be a useful one.

The development is as follows.
Section 2 gives a self-contained account of continuous flexibility and rigidity, and
vanishing flexibility and one-sided flexibility is proven for various periodic  infinite frameworks. Forms of flexibility, such as bounded flexes, square summable flexes, summable flexes and vanishing flexes are defined and determined for some specific examples. Sufficient conditions are obtained for the existence of a smooth flex and a flex extension problem for generic finite frameworks is posed. A positive resolution of this problem would provide a natural extension  of  Laman's theorem to infinite frameworks. In Section 3 we consider infinitesimal theory for general infinite frameworks and determine a number of rigidity operators and their flex and stress spaces. (The topic is taken up in more detail for crystal frameworks in Section 5.)
In Section 4 we consider (abstract) crystal frameworks in two or three dimensions. These are generated by a motif and a discrete translation group. Various forms of deformations are considered, such as strict periodic flexibility, flow-periodic flexibility and flexes with reduced periodicity and symmetry.
Also we indicate the flat torus model for crystal frameworks and recent results of Ross and Whiteley in this direction, including a periodic analogue of Laman's theorem.
In the final section we consider the matricial symbol function approach and various examples. In particular
we determine the (unit cell) infinitesimal wave flex multiplicities
for the kagome net framework by factoring the determinant of the matricial symbol function.

\textit{Acknowledgments.} Some of the developments here have benefited from discussions and communications with Robert Connelley, Patrick Fowler, Simon Guest, Elissa Ross and Walter Whiteley
during and following the Summer Research Workshop on "Volume Inequalities and Rigidity",
organised by K\'aroly Bezdek,
Robert Connelly,
Bal\'azs Csik\"os and
Tibor Jord\'an,  at the Department of Geometry, Institute of Mathematics,
Eotvos Lorand University in July 2009.

\section{Infinite Bar-joint Frameworks}

In this section we give a self-contained rigourous development of infinite frameworks and examine the nonlinear aspects of their flexibility by continuous deformations and their associated rigidity.
In the next section we  consider infinitesimal flexibility and rigidity in
a variety of forms.

\subsection{Continuous flexibility and rigidity} We first define continuous flexes and continuous rigidity. The latter  means, roughly speaking, that the framework admits  no proper deformations that preserve the edge lengths.
The definition below gives straightforward generalisations of  terms used for finite frameworks. In that case we note that a continuous flex is often referred to as a finite flex
while in engineering models it is referred to as a finite mechanism.

Unless we specify otherwise
we shall assume that the frameworks under consideration are \textit{proper} in that the
framework points do not lie on a hyperplane in the ambient space $\bR^d$ and that the framework edges $[p_i, p_j]$ have nonzero lengths $|p_i-p_j|$.

\begin{definition}
Let  $(G, p)$ be an infinite framework in $\bR^2$, with
connected abstract graph  $G=(V, E)$,  $V = \{v_1, v_2, \dots
\}$ and $p=(p_1,p_2, \dots )$.

(a) A base-fixed continuous
flex, or, simply, a flex of $(G,p)$, is a function \\ $p(t) = (p_1(t), p_2(t),
 \dots )$
 from  $[0,1]$ to $\prod_V \bR^2$ with the following
 properties;

(i) $p(0) = p$,

 (ii) each coordinate function $p_i: [0,1] \to \bR^2$ is  continuous,

(iii) for some \textit{base edge} $(v_a,v_b)$ with $|p_a - p_b| \neq 0$,
$p_a(t) = p_a(0)$ and $p_b(t) = p_b(0)$ for all $t$,

(iv) each edge distance is conserved; $|p_i(t) - p_j(t)| = |p_i(0)
- p_j(0)| $ for all edges $(v_i , v_j)$, and all $t$, and


(v) $p(t)\neq p$ for some $t \in (0,1]$.

(b) The framework $(G, p)$ is  \textit{flexible}, or more precisely,
\textit{continuously flexible}, if it possesses a  base-fixed   continuous
flex.

(c) The framework $(G, p)$ is \textit{rigid}, or continuously
rigid, if it is not flexible.
\end{definition}

Similarly one defines base-fixed continuous flexes  and continuous rigidity
for  proper frameworks in $\bR^d$ by replacing a base edge by an appropriate set of framework points with maximal affine span.

The simplest kind of continuously rigid framework in the plane is one which
is a union  of continuously rigid finite frameworks. In particular
the following
theorem follows simply from the theorem of Laman indicated in the introduction.

\begin{thm} Let $G$ be a connected graph which is the union of a sequence of
finite Laman graphs. Then every
generic  realisation $(G,p)$  in the plane is continuously rigid.
\end{thm}

By \textit{generic}, or, more precisely, \textit{algebraically generic},  we mean, as is
usual, that the coordinates of any finite set of framework
points is algebraically independent over the rational numbers.
Unlike the case of finite frameworks
it is possible to construct two generic frameworks with the same abstract graph
one of which is flexible and one of which is rigid.
Accordingly it seems appropriate to formulate the following definition to extend the usual
usage.

\begin{definition}
An infinite simple connected graph $G$ is said to be rigid, or generically rigid, for two  dimensions, if every generic framework $(G,p)$ in the plane is rigid.
\end{definition}

Note that if $G$ is rigid and if $H$ is a containing graph for
$G$ with the same vertex set then every generic framework $(H,p)$ in the plane is rigid.

It seems likely that the converse to the theorem above holds. That is, if $H$ does not contain
 a \textit{sequentially Laman} graph (in the sense below)  with the same vertex set,  then
 $H$ is not generically rigid. We comment more on this later in Section 2.6.

Rigidity and flexibility are  properties of the entire framework and  it is such
entire features and their inter-relationships that are of primary interest in what
follows. One would like to understand the relationship with
small scale or local structure, such as local counting conditions and local connectivity. Additionally, as above, one would like to relate entire properties to
sequential features that pertain to an exhausting chain of finite subframeworks and for this
the following definition is helpful.

\begin{definition}
If $P$ is a  property for a class of finite, simple, connected graphs (resp. frameworks) then
a graph $G$ (resp. framework  $\G= (G,p)$) is \textit{sequentially $P$} or $\sigma$-$P$ if $G$
is the union of graphs in some increasing sequence of vertex induced finite subgraphs
$G_1 \subseteq G_2 \subseteq \dots ,$ and
each graph $G_k$ (resp. framework $(G_k ,p)$) has property $P$.
\end{definition}

For example,
we may refer to an infinite graph as being $\sigma$-Laman, or $\sigma$-Laman$-1$
and an infinite framework as being $\sigma$-rigid.
To say that an infinite framework is $\sigma$-flexible, or sequentially flexible, is rather vacuous
since it usually prevails for trivial reasons.
(One can construct countably infinite edge complete frameworks that may fail to be so
but all our examples are sequentially flexible.) The more interesting  property is the
failure  of {sequentially rigidity} as in the following definition.

\begin{definition}
A framework $\G$ is said to have \textit{vanishing flexibility} if it is
continuously rigid but not $\sigma$-rigid.
\end{definition}

An important topic in finite rigidity frameworks is  that of \textit{global rigidity}, also
called unique rigidity. This holds, in the two dimensional setting,
if there is, up to congruency, only one embedding of the framework in $\bR^2$. This implies,
for example, that
these frameworks admit no foldings, and indeed have  unique
diagrams up to rotations and reflection.  One can extend the term to infinite
frameworks but we do not consider this issue here at all.
One might be tempted to say that a rigid infinite framework, especially one with vanishing flexibility,
is globally rigid, but we refrain from doing so because of conflict with this usage.

The term "global" for global rigidity is natural since rigidity for finite frameworks
is equivalent to the "local" property that "nearby equivalent frameworks are congruent". That is, if  there  exists $\epsilon >0$ such that if $(G,p')$ is a finite framework with $|p_i-p_i'|<\epsilon$, for
all $i$, and if $(G,p')$ is \textit{equivalent}
to $(G,p)$, in the sense that corresponding edges have the same length, then
$(G,p')$ and  $(G,p)$ are {congruent}. See, for example,   Gluck \cite{glu} and Asimow and Roth \cite{asi-rot}.

 An infinite simple graph $G$ is \textit{locally finite} if for every vertex $v$ there are finitely many incident
edges. Amongst such graphs are those for which there is an upper bound to the degree of the vertices, as in the case of the graphs of crystal frameworks. Within this class a graph $G$ is said to be \textit{$r$-regular} if every vertex has degree $r$. We remark that the
 theory of tilings provides a wealth of examples of planar frameworks which are $4$-regular.

\begin{rem} {\rm In what follows we consider only locally finite frameworks. Without this assumption
it is possible to construct quite wildly flexing planar linkages.
In fact, given a continuous function $f: [0,1] \to \bR^2$
one can construct
an infinite linkage, in the sense of the definition below, and a base-fixed flex $p(t)$
with a motion $p_v(t)$ for a particular vertex $v$ that is equal to $f(t)$. This includes the possibility of space filling curves. This is a consequence of a continuous analogue
of a well known theorem of Kempe which asserts that any finite algebraic curve in
the plane can be simulated by a finite linkage.
For more details see Owen and Power \cite{owe-pow-kempe}}.
\end{rem}

\subsection{Linkages}
The removal of a framework edge from a rigid framework  may result in flexibility
which is, roughly speaking, of a one-dimensional nature. We reserve the term \textit{linkage} for such a mathematical object, which we formally specify in the next definition.
We remark that finite frameworks are also referred to as linkages,
particularly when they are flexible, perhaps with several degrees of freedom, but this should not cause confusion.

 A {\it two-sided  continuous flex} $p(t)$ of $(G,p)$ is defined as above but for the replacement of $[0,1]$ by $[-1,1]$. The following formal definition of an infinite linkage
 reflects the fact that the
 initial motion of a base-fixed
linkage is uniquely determined by the angle change at any flexible joint.

\begin{definition}
A linkage in $\bR^2$ is a finite or
infinite connected framework $\G = (G,p)$ in $\bR^2$
for which there exists a  continuous two-sided base-fixed flex $p(t)$ with
framework edges $[p_i, p_j],[p_j, p_k]$ such that
the cosine angle function
\[
g(t)=\langle p_i(t)-p_j(t), p_k(t)-p_j(t)\rangle
\]
is strictly increasing on $[-1,1]$, and such that
 $p(t)$ is the unique two-sided flex $q(t)$ of $\mathcal{G}$ with
$q_l(t)=p_l(t)$, for $l=i,j,k$.
\end{definition}

Many interesting finite linkages were considered in the nineteenth century in connection
with mechanical linkages. See, for example, Kempe \cite{kem}.
Note however that  the definition is liberal
in allowing coincident joints and self-intersecting flexes.
Also the definition refers to local deformation behaviour and this  does not rule
out the possibility of bifurcations occurring in a parameter extension of the given flex.

It is a simple matter to construct diverse infinite linkages
by   tower constructions or progressive assembly.
(See, for example, the Cantor tree frameworks of \cite{owe-pow-3}.)
However, some such constructions lead to frameworks
with vanishing flexibility and so are not linkages in this case.
 An elementary illustration is given
in Figure 4 wherein a two-way infinite  rectangular strip linkage
is augmented by  adding
flex-restricting cross braces in an alternating fashion.
 If the brace lengths tend to
the diagonal length from above then the  infinite framework is  rigid.
Evidently in this case the triangle inequality is playing a role in isolating one real solution to the \textit{solution set} $V(G,p)$ defined below. One can also construct
examples in which this isolation is less evident, with all joint angles bounded away from zero and $\pi /2$ for example.

A more interesting and subtle form of vanishing flexibility is due to
progressive flex amplification rather than local flex restrictions.
Roughly speaking, if a small flex is initiated at a particular joint and the flex
propagates in some amplifying manner, then the triangle inequality at some far remove
may prohibit any further increase. If the framework is infinite then
no local joint flex may be possible at all.
The strip framework of concatenated levers in Figure 5 gives an example
where the amplification is evident, while  the rigidity of the
strip framework in Figure 6 and the trapezium strip in Figure 7 is less evident.
The lever framework has a natural infintesimal flex, in the sense of Section 3, which is unbounded. The corresponding flexes of Figures 6 and 7 however, are bounded with amplification unfolding as a second order effect.
This is proven in the next subsection.


\begin{figure}[h]\label{f:rectstrip}
\centering
\includegraphics[width=8cm]{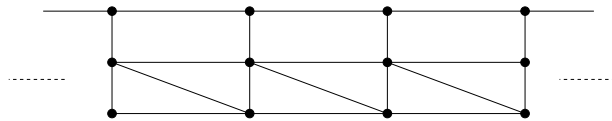}
\caption{An infinite rectangle strip linkage.}
\end{figure}

\begin{center}
\begin{figure}[h]\label{f:rectstripbraced}
\centering
\includegraphics[width=8cm]{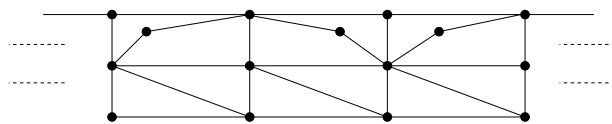}
\caption{A restricted rectangle strip.}
\end{figure}
\end{center}

\begin{center}
\begin{figure}[h]\label{f:traptilt}
\centering
\includegraphics[width=9cm]{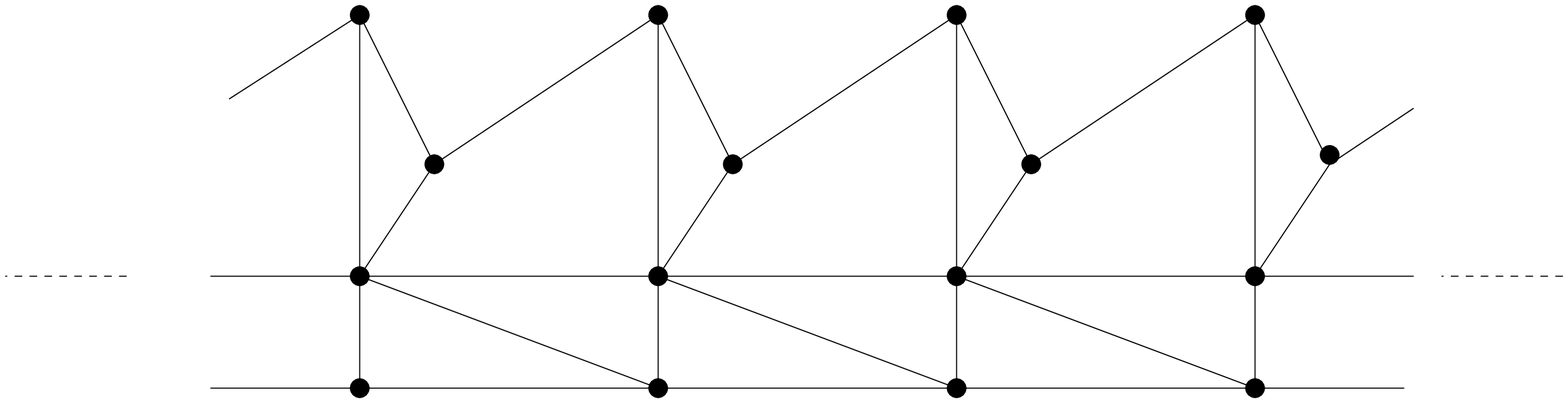}
\caption{Rigid but not $\sigma$-rigid.}
\end{figure}
\end{center}

\begin{center}
\begin{figure}[h]\label{f:kagomestrip}
\centering
\includegraphics[width=9cm]{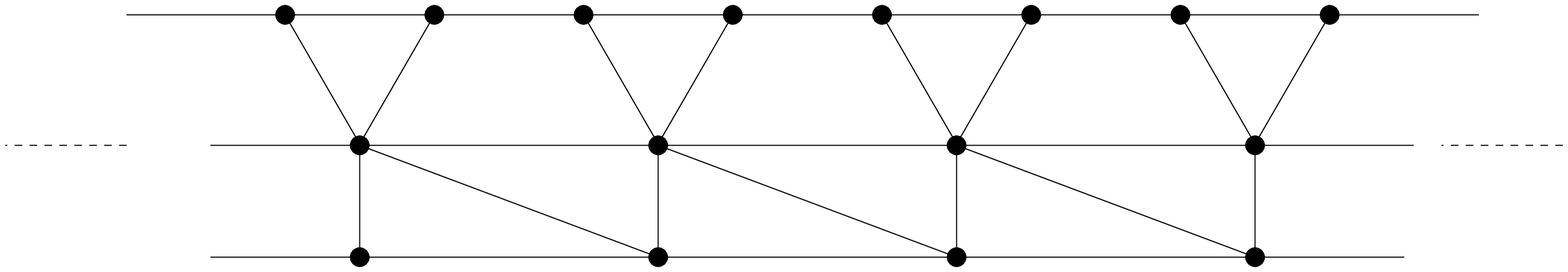}
\caption{Rigid but not $\sigma$-rigid.}
\end{figure}
\end{center}

\begin{center}
\begin{figure}[h]\label{f:kagomestrip1}
\centering
\includegraphics[width=9cm]{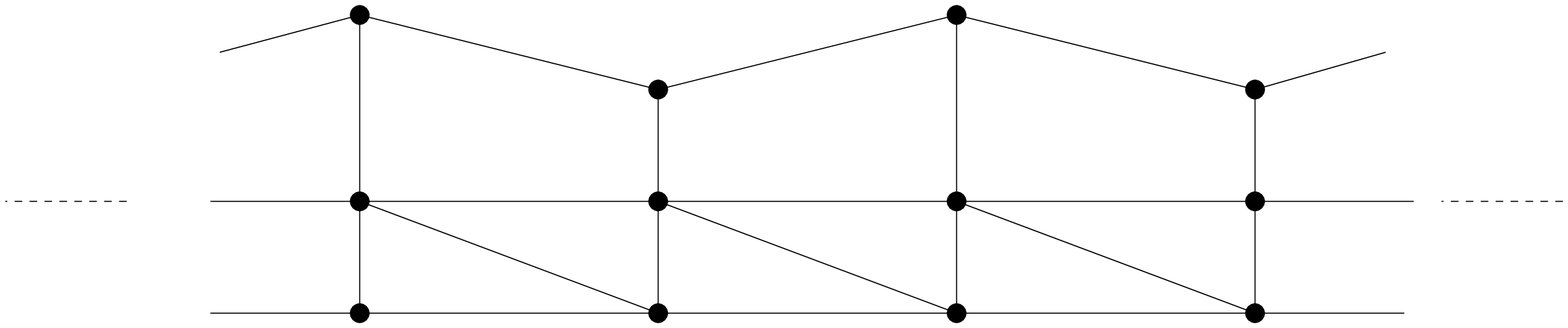}
\caption{Rigid but not $\sigma$-rigid.}
\end{figure}
\end{center}

\begin{center}
\begin{figure}[h]\label{f:kaghalfstrip}
\centering
\includegraphics[width=8cm]{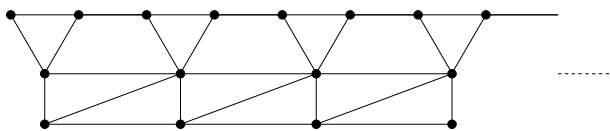}
\caption{A periodic half-strip which is only right-flexible.}
\end{figure}
\end{center}

It is a straightforward matter to incorporate the vanishing flexibility
of the strips above as subframeworks
of a $\bZ^2$-periodic framework. This process is indicated in
Figures 9, 10 and 11 below, where the infinite frameworks are determined
as the periodic extensions of the given unit cell.
Figure 9 shows a linkage formed as a "fence lattice" composed
of infinite  horizontal and vertical $\sigma$-rigid bands. Figure 10
shows an analogue where the infinite  bands have been replaced by
rigid strip frameworks. Figure 11 is an elaboration
of this in which cross braces have been introduced to remove
the flexibility. (Only one edge is needed  for this whereas the example
given is periodic.) The additional  degree 2 vertex in the cell ensures
 that the framework is not $\sigma$-rigid, while the infinite bands remain vanishingly rigid.

Let us note that for the framework in Figure 11,
with its curious mixture of rigidity and flexibility, one can add any finite number of additional degree 2 vertices without changing the rigidity of the framework.
In particular we have a construction that proves the following proposition.

\begin{center}
\begin{figure}[h]
\centering
\includegraphics[width=4cm]{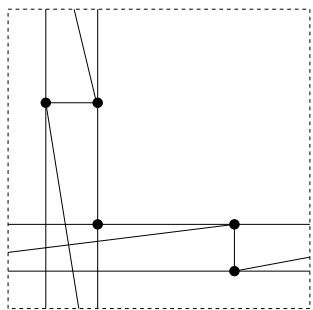}
\caption{Unit cell for a  "fence lattice" linkage.}
\end{figure}
\end{center}

\begin{center}
\begin{figure}[h]
\centering
\includegraphics[width=4cm]{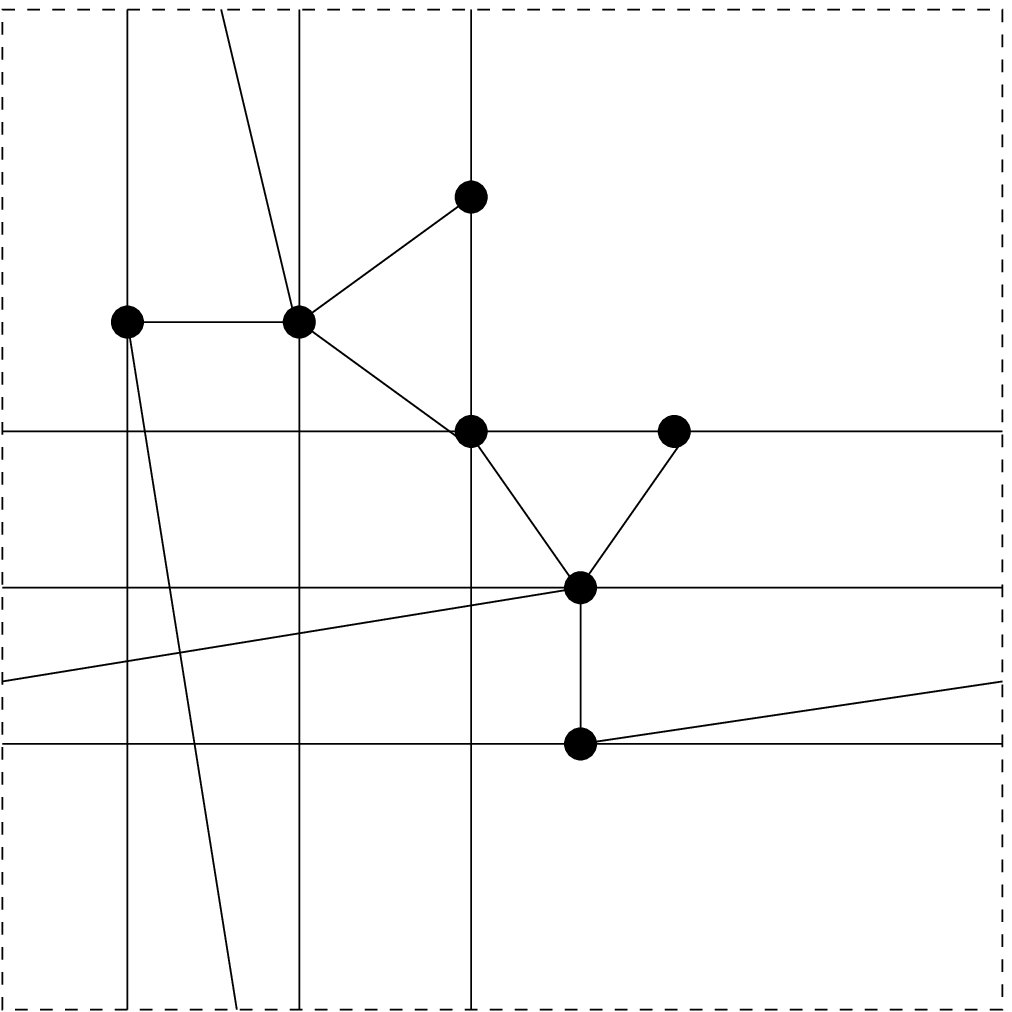}
\caption{Unit cell for  a modified fence lattice linkage.}
\end{figure}
\end{center}

\begin{center}
\begin{figure}[h]
\centering
\includegraphics[width=4cm]{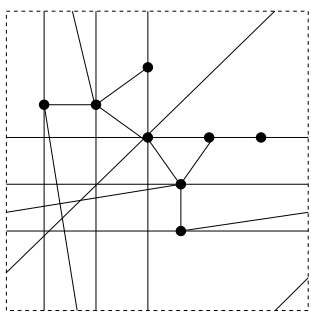}
\caption{Unit cell for a rigid periodic framework which is not $\sigma$-rigid.}
\end{figure}
\end{center}

\begin{prop}
Let $c>2$. Then there is a $\bZ^2$-periodic framework in $\bR^2$ which is rigid, which is not $\sigma$-rigid and for which the average vertex degree is less than $c$.
\end{prop}

One can readily extend this fanciful idea in various ways to obtain
such structures in higher dimensions. For example, start with a one-dimensionally  periodic $\sigma$-rigid girder in 3D and augment it with
trapezium "tents" of alternating height to creates vanishingly rigid girders.
Also periodically interpolate any number of degree two vertices into the tent top edges
without removing the vanishing flexibility.
Join infinitely many such component girders periodically at appropriate tent-top \textit{edges} to create a fence framework and add linear jointed cross braces to create, finally,  a 2D periodic grid which is continuously rigid in 3D, which is not $\sigma$-rigid and which has average coordination number arbitrarily close to two.

\subsection{Relative rigidity and the extension of flexes.}

For a finite or infinite framework $\G= (G,p)$ in $\bR^2$ define the function
\[
f_G : \prod_V \bR^2 \to \prod_E \bR, \quad f_G(q) = (|q_i-q_j|^2)_{e=(v_i,v_j)}.
\]
This is the usual \textit{edge function} of the framework and depends only on the abstract graph
$G$.

\begin{definition}
The solution set, or { configuration space}, of a framework $\G=(G,p)$, denoted  $ V(G,p)$, is the set  $ f_G^{-1}(f_G(p)). $ This is the set of all framework vectors $q$ for $G$ that satisfy the distance constraints equations
\[
 |q_i-q_j|^2 =  |p_i-p_j|^2, \mbox{  for all edges } {e=(v_i,v_j)}.
\]
\end{definition}

In general the  solution set of an infinite framework  need not be a real algebraic variety
even when it is "finitely parametrised". In less wild situations
it can be useful to relate $ V(G,p)$ to the algebraic variety
$V(H,\pi_H(p))$ associated with a finite subgraph $H$ of $G$, or with an elementary subgraph
such as a tree, or even a set of vertices.

\begin{definition}
An infinite  bar-joint framework $(G,p)$ in $\bR^d$ is rigid over a subframework $(H, \pi_H(p))$ if
 every continuous flex of $(G,p)$ which is constant valued on $(H, \pi_H(p))$ is constant.
Similarly, if $H$ is a subgraph of a  countable connected simple graph $G$
then $G$ is \textit{rigid over
$H$}, or generically rigid over $H$ if, for every generic frameworks $(G,p)$,  every continuous flex of $(G,p)$ which is constant-valued on $(H, \pi_H(p))$ is constant.
\end{definition}

We may also form the following associated notions.

\begin{definition}
An infinite framework $(G,p)$ in $\bR^d$ is  finitely determined  if it is rigid over $(H,p)$ for some finite subgraph $H$ and is finitely flexible if it is flexible and finitely
determined.
\end{definition}


Finite flexibility in the sense above is a strong property in which
paths from $p$ in the solution set
$V(G,p)$ are determined near $p$ by the finite algebraic variety $V(H,p)$.
Note that  the term "infinitely flexible" is not appropriate  to describe a flexible framework
which is not finitely flexible since it is possible to construct
linkages, in our formal sense, which are not finitely flexible.
This is the case for the periodic framework in  Figure 12 which is a linkage because of
partial vanishing flexibility.

\begin{center}
\begin{figure}[h]\label{f:kagstripflex}
\centering
\includegraphics[width=8cm]{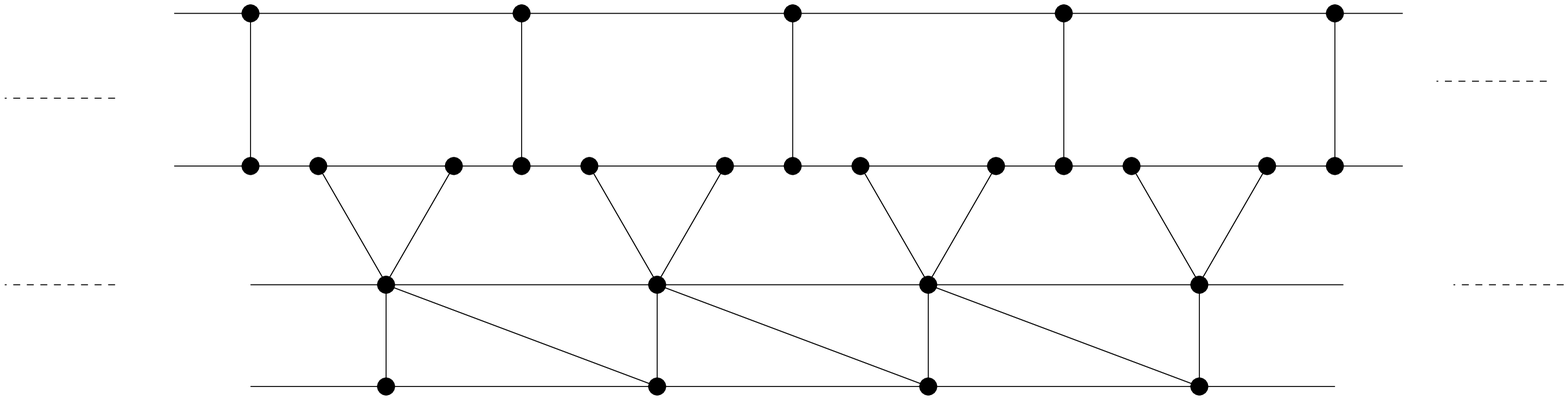}
\caption{A linkage which is not finitely determined.}
\end{figure}
\end{center}

An important class of frameworks which appear in
 mathematical models  are those that are distance regular.

\begin{definition}
 A framework $\G= (G,p)$ in $\bR^d$ is distance-regular
if there exist $0< m < M$ such that for all edges $(i,j)$,
\[
m < |p_i-p_j| \leq M.
\]
\end{definition}

For such a framework $(G,p)$ it is natural to consider
the nearby frameworks with the same graphs but with slightly perturbed
framework points (and therefore edge lengths). If a property holds for all such frameworks,
for some perurbation distance $\epsilon$ then we call such a property
a \textit{stable property} for the the framework.

Formally,  an \textit{$\epsilon$-perturbation} of a distance regular framework
 $\G= (G,p)$ is a framework  $\G'= (G,p')$ for which
$ |p_i-p_i'| < \epsilon$, for all corresponding framework points.
Recall that a finite framework in $\bR^d$ is said to be \textit{$\epsilon$-rigid} if
it is congruent to every \textit{equivalent} $\epsilon$-perturbation.
Let us say that a general
framework is \textit{perturbationally rigid} if it is {$\epsilon$-rigid}
for some $\epsilon$. It is a well-known
fact that perturbational rigidity
and rigidity are equivalent in the case of algebraically generic
finite frameworks \cite{asi-rot}, \cite{glu}. However, it is straightforward to see
that this equivalence thoroughly fails for general infinite frameworks (see \cite{owe-pow-3}).

\begin{definition}
Let $\G$ be a distance-regular framework. Then $\G$  is stably  rigid (resp stably flexible)
if it is rigid (resp. flexible) and for sufficiently small $\epsilon >0$
every  $\epsilon$-perturbation of $\G$
is rigid (resp. flexible).
\end{definition}

Likewise, if $P$ is any particular property of a distance-regular infinite framework
then we may say that $\G$ is \textit{stably} $P$ if, for some $\epsilon > 0$, the property
$P$ holds for all $\epsilon$-perturbations.

\begin{prop}
The periodic trapezium strip frameworks, with alternating unequal heights $a>b>o$, are rigid.
In particular the rectangle strip linkage (of Figure 7) is not stably flexible.
\end{prop}

\begin{center}
\begin{figure}[h]\label{}
\centering
\includegraphics[width=6cm]{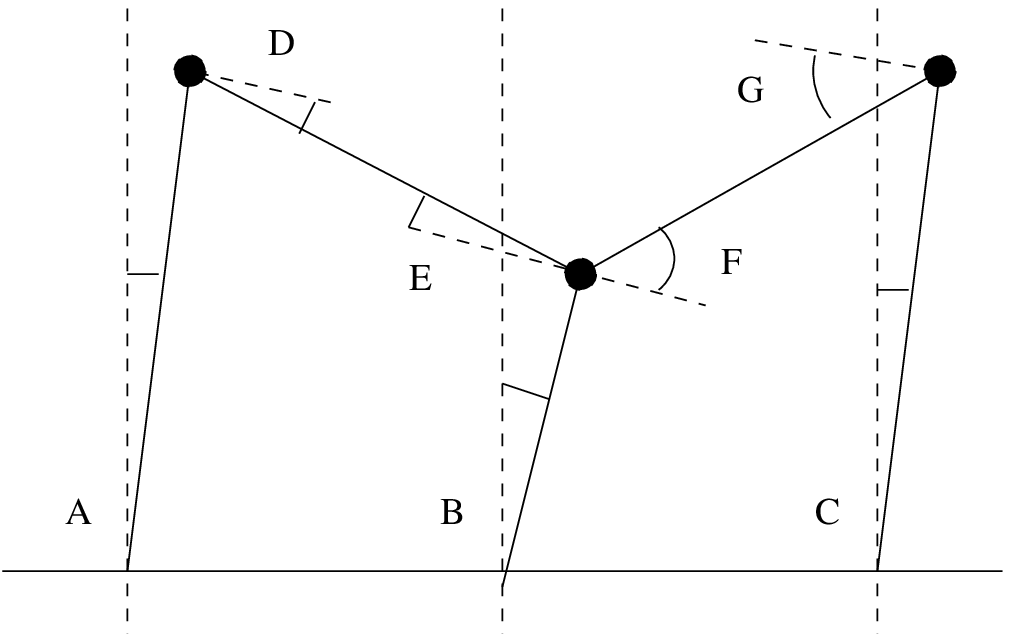}
\caption{}
\end{figure}
\end{center}

\begin{proof}
Figure 13 shows the displacement of a double trapezium to the right. Let the three vertical bar lengths be $a,b$ and $a$ units with $a>b>0$. The displaced position has angles $A,B,C$ at the base line and angles $D,E,F,G$ occurring relative to the trajectory tangents of the displaced vertices. For a subsequent incremental change $\delta A$, with resulting incremental changes $\delta B, \delta C,\delta D,\delta E,\delta F,\delta G$ we can see
from simple geometry that to first order
\[
a(\delta A \cos D) = b(\delta B\cos E).
\]
Suppose now that $A$ is regarded as a specialisation of the input angle $ \alpha$ with resulting output angle
$\beta = \beta (\alpha)$, so that at $\alpha = A$ we have $\beta(\alpha) = B$.
Then
\[
\frac{d\beta}{d\alpha}|_{\alpha =A} = \lim_{\delta A \to 0}\frac{\delta B}{\delta A} = \frac{a}{b}\frac{\cos D}{\cos E}.
\]
Similarly, with angle $C$ regarded as the  output angle $\gamma (B) $ for the angle transmission function
$\gamma = \gamma(\beta)$, we have
\[
\frac{d\gamma}{d\beta}|_{\beta =B} = \frac{b}{a}\frac{\cos F}{\cos G}
\]
and so
\[
\frac{d\gamma}{d\alpha}|_{\alpha =A}= \frac{d\gamma}{d\beta}|_{\beta =B}\frac{d\beta}{d\alpha}|_{\alpha =A} =\frac{\cos D}{\cos E}\frac{\cos F}{\cos G}.
\]
Note that since $B>A$ we have also $D>E$, and since $B>C$ we have $F>G$, from which it follows that both ratios above are less than one. Thus certainly
$0 < \gamma'(\alpha) < 1 $ for $0<\alpha <\alpha_1$ where $\alpha_1$ is the first positive angle for which $\gamma'(\alpha_1) =0.$

It follows, from the mean value theorem, that the double trapezium angle transmission function is an increasing differentiable function with
$$\gamma(0) =0, \quad 0 < \gamma(\alpha) < \alpha, \quad\mbox{for} \quad 0<\alpha < \alpha_1.$$
It follows immediately that the right-semi-infinite trapezium strip is right flexible.

We let $\lambda = \gamma(\alpha_1)$, which we refer to as the locking angle. Note the second trapezium of the double admits no increase of this angle. In view of the above
we have $\lambda < \alpha_1$.

Suppose now that $\G$ is the two-way-infinite trapezium strip, with $a \neq b$.
Let $p(t)$ be a flex and suppose that for a fixed framework edge with length $a$ the angle $A$ is greater than zero for some time $t_1 > 0$ and that $t_1$ is the first such time. Then certainly $0<A<\lambda$. Note that $A_{-n}= \gamma^{-n}(A), n=1,2, \dots  $ are the angles of the edges of length $a$, counted off to the left. In view of the function dominance $0 < \gamma(\alpha) < \alpha$
it follows that $A_{-n} >\alpha_1$ for some $n$, which is a contradiction.
\end{proof}

The argument above also shows that the semi-infinite trapezium strip framework
of Figure 7 has a continuous
flex but no two-sided flex.

\begin{rem} {\rm  It seems to be of interest to analyse strip frameworks in further detail. For example, a trapezium strip framework is not stably rigid,
despite the apparent "robustness" of the argument above.
To see this use surgery in the following way. Remove one cross bar, then push the rightmost semifinite strip to the right, by an angle perturbation $A= \epsilon >0$. Now insert a replacement bar of the required length. One can flex the resulting structure towards  the left to restore the position of the right hand strip. Indeed, this is all the flexibility
the framework has.
The possible flex of an $\epsilon$ perturbation, such as the one described, seems to be of order $\epsilon$ and
so there does seem to be "approximate rigidity".
}
\end{rem}

\begin{rem}{\rm  Consider a periodic  trapezium grid framework $\G_{trap}$
obtained by perturbing $\G_{\bZ^2}$ by adding a fixed small positive value to the $y$ coordinate of the framework points $p_{ij}$ for the odd values of $i$ and $j$.
It can be shown that this framework is rigid over any linear subframework.
This contrasts with the grid framework itself which is freely flexible over its $x$ and $y$ axes in the following specific local sense: all sufficiently small flexes of the subframework extends to a flex of the whole framework.
On the other hand note that Theorem 4.4 shows that
$\G_{trap}$ is deformable. }
\end{rem}

\subsection{Forms of flexibility}
It seems to be a fundamental and interesting issue to determine the
ways in which infinite bar-joint frameworks are rigid or continuously
flexible. In this section we
give some further definitions, we give  sufficient conditions for the existence of a proper flex and we contemplate a plausible infinite framework version of Laman's theorem.

Flexes are often infinitely differentiable or smooth in the sense of the following formal definition. This is the case for example, for  the "alternation" flexes of $\G_{sq}$ and $\G_{kag}$.

\begin{definition}
A continuous base-fixed two-sided flex $p(t): t \in [-1,1]$ of a framework
$(G,p)$ in $\bR^d$ is a
\textit{smooth flex} if each coordinate function $p_i(t)$ is infinitely differentiable.
\end{definition}

The smoothness of a flex is a local requirement whereas the following terms
impose various increasing forms of global constraint. In particular rotational flexes of infinite
frameworks with unbounded diameter are not bounded flexes, while a translational flex of
an infinite framework is a bounded flex  but is not a vanishing flex.
Adopting a term that has been used in applications \cite{goodwin2008} we refer to flexes
which are not bounded as \textit{colossal flexes}.

\begin{definition}
A continuous flex $p(t)= (p_k(t))_{k=1}^\infty , (t\in [0,1])$ of an infinite framework $(G, p)$ in $\bR^d$ is said to be

(i) a bounded flex  if for some $M>0$ and every $k$ and $t$,
 $$|p_k(t) - p_k(0)| \leq M,$$

(ii) a colossal flex  if it is not bounded,

(iii) a  vanishing flex if $p(t)$ is a bounded flex
and if the maximal displacement
\[
\|p_k - p_k(0)\|_\infty = \sup_{t\in [0,1]} |p_k(t) - p_k(0)|
\]
tends to zero as $k \to \infty$,

(iv) a square-summable flex if
\[
\sum_{k=1}^\infty \|p_k-p_k(0)\|_\infty^2 < \infty,
\]

(v) a summable flex if
\[
\sum_{k=1}^\infty \|p_k-p_k(0)\|_\infty < \infty,
\]

(vi) an internal flex if for all but finitely many $k$ the function $p_k(t)$ is constant.
\end{definition}

Also we say
that $(G,p)$ has a deformation  (resp. bounded or vanishing deformation) if it has a base-fixed flex $p(t)$ (which is bounded or vanishing).


\begin{definition}
A connected infinite locally finite proper framework in two or three
dimensions is boundedly rigid (resp. summably rigid, square-summably rigid, smoothly rigid, internally rigid) if there is no deformation, that is, no base-fixed proper continuous flex, which is bounded (resp. summable, square-summable, smooth, internal).
\end{definition}

\subsection{Sufficient conditions for flexibility}
There is a sense in which vanishing flexibility
is almost the only obstacle to the existence of a flex of a framework
all of whose finite subframeworks are flexible. More precisely, in the hypotheses of the next theorem we assume that there are two distinguished framework vertices, $p_1,
p_2,$ such that any finite subframework $(H,p)$ containing $p_1, p_2$ has a flex which
properly separates this pair in the sense of condition (ii) below.
The additional requirement needed is that there is a family of flexes of the
finite subframeworks whose restrictions to any given subframework $(H,p)$
are uniformly smooth in the sense of condition (i). Note that the constant here
depends only on $(H,p)$ and indeed the resulting flex may of necessity be
a colossal  flex.

\begin{thm} \label{t:aathm}Let $(G,p)$ be an infinite locally finite
framework in $\bR^d$
with a connected graph, let
$$(G_1,p)\subseteq (G_2, p)
\subseteq \dots ,$$ be
subframeworks, determined by finite subgraphs $G_r = (V_r, E_r)$
with union equal to  $G$ and let $v_1, v_2$ be  vertices in $G_1$. Suppose
 moreover that  for each $r=1,2,\dots ,$ there is a base-fixed smooth flex
 $p^{r}(t)= (p^r_k(t))_{k=1}^{|V{(G_r)}|}$  of $\G_r = (G_r,p)$ such that

(i) for each finite framework $\G_l$ the set $\F_l$ of restriction flexes
$$\{p^{r}(t)|\G_l : r\geq l \}$$
have uniformly bounded derivatives, that is, there are constants $M_l, l=1,2,\dots ,$ such that
\[
|\frac{d}{dt}p_k^{r}(t)| \leq M_l \quad \mbox{  for  } r\geq l, v_k \in V_l,
\]
(ii) the framework points $p_1, p_2$ are uniformly separated by each flex $p^r(t)$
in the sense that
\[
|p^{r}_1(1)- p^{r}_2(1)| - |p^{r}_1(0)-p^{r}_2(0)| \geq c
\]
for some positive constant $c$.

Then
$(G,p)$ has a deformation.
\end{thm}

Note that it is essential that the separated vertices of condition (ii)
are the same for each subgraph.
To see this note that the two-way infinite trapezium strip framework considered in Figure 7 has smooth
 deformations on each of its finite strip subframeworks, each of which "separates" some two vertices (at the end of the strip) by a fixed positive distance. Nevertheless the infinite strip fails to have a deformation.

\begin{proof} For $l=1,2, \dots ,$
let $X_l$ be the space of continuous functions from $[0,1]$ to $\bR^{d|V_l|}$
and note that the family $\F_l$, by the hypotheses, is an equicontinuous family in $X_l$. Moreover with respect to the supremum norm $\F_l$ is a bounded set.
By the Ascoli-Arzela theorem (see \cite{rud} for example) $\F_l$ is  precompact  and in particular for $\F_1$
there is a subsequence $r_1, r_2, \dots $ such that the restrictions $p^{r_k}|\G_1$
converge uniformly in their (finitely many) coordinates to a flex $q^1$ of $\G_1$.
Relabel the sequence $(p^{r_k})$ as $(p^{(k,1)})$. The restrictions of these flexes
to $\G_2$ similarly have a convergent subsequence, say $(p^{(k,2)})$, and so on.
From this construction select the diagonal subsequence $(p^{(k,k)})$. This converges
coordinatewise uniformly to a coordinatewise continuous  function
\[
q:[0,1] \to \bR^d \times \bR^d \times ....
\]
Since the restriction of $q$ to every finite subframework is a flex, $q$
satisfies the requirements of a flex of $(G,p)$, except possibly the properness requirement (v) of Definition 2.1. In view of (ii) however, $q(0)\neq q(1)$ and so $q$
 is a continuous flex of $(G,p)$.
\end{proof}

\begin{rem}{\rm
Computer simulations provide evidence for the fact that small random perturbations of $\G_{\bZ^2}$ yield frameworks that are flexible. That  is it seems that $\G_{\bZ^2}$
 is stably flexible. It would be interesting if the theorem above could assist in a proof of this.}
\end{rem}

\subsection{Flex extensions and generic rigidity}
Let us recall a version of Laman's theorem.

\begin{thm}
Let $(G,p)$ be an algebraically  generic finite framework. Then $(G,p)$ is infinitesimally
rigid  if and only if the graph $G$ has a vertex induced subgraph $H$, with $V(H) = V(G)$, which is maximally
independent in the sense that $2|V(H)|=|E(H)| +3$ and $2|V(H')|\geq |E(H')| +3$ for every subgraph $H'$ of $H$.
\end{thm}

 For convenience we refer to a maximally independent finite graph as a Laman graph.
 We remark that any Laman graph can be obtained from a triangle graph by a sequence of
 moves known as Henneberg moves. The first of these adds a new vertex with two connecting
 edges while the second breaks an edge into two at a new vertex  which is then connected by a new edge to another point of the graph.

 Now let $G$ be an infinite graph which contains a subgraph $H$ on all the vertices of $G$ and suppose that $H$ is $\sigma$-Laman.  In view of Laman's theorem every algebraically generic realisation of $H$ (and hence $G$)  in the plane is $\sigma$-rigid and so continuously rigid. Is the converse true ? That is, if every generic realisation of
 an infinite graph $G$ is rigid does $G$ necessarily contain a $\sigma$-Laman subgraph
 with the same vertex set.

To see this one needs to show is that if $G$ is $\sigma$-(Laman-1) and not $\sigma$-Laman, then there exists a vertex generic  realisation $(G,p)$  in the plane which has a continuous flex. That is, we want to build up a flex of the infinite structure by adding new vertices and edges, in the least handicapping way, to allow all, or most of the flex
of an initial finite subgraph
to be extended.

Alternatively, and more explicitly, suppose that one starts with
a generic connected Laman-1 framework $(G_1, p)$ with $n$ vertices
and  an infinite sequence of Henneberg move "instructions". These instructions
 yield a unique infinite graph. Is it possible to choose associated framework
 points  $p_n, p_{n+1},..$ so judiciously that some (perhaps small)
flex of $(G_1, p)$  extends fully to each successive finite extension framework (and hence the infinite framework) ?

\section{Rigidity Operators and Infinitesimal Rigidity}
In previous sections we have
considered  some variety in
the \textit{nature} of continuous flexes $p(t)$ and how they might distinguished.
A companion consideration
is the analysis of various spaces of infinitesimal flexes. This gives
insight into continuous flexes since the derivative $p'(0)$ of a differentiable flex $p(t)$ is an infinitesimal  flex.

Here we give an operator theory perspective for an infinitesimal theory of
infinite frameworks in which the rigidity matrix
$R(G,p)$ is viewed as a linear transformation or linear operator
between various spaces. The domain space contains a space of infinitesimal flexes,
which lie in the kernel of the rigidity operator, while the
range space contains a space of {self-stress vectors} namely those in the kernel of the transpose of $R(G,p)$.

\subsection{Infinitesimal rigidity and the rigidity matrix.}
Recall that for a finite framework $(G,p)$ in $\bR^d$ with $n=|V|$
an \textit{infinitesimal flex} is a vector $u = (u_1,\dots ,u_n)$
in the vector space
 $\H_v = \bR^d \oplus \dots \oplus \bR^d $ such that the orthogonality
relation
$\langle p_i-p_j,u_i-u_j\rangle =0$
holds for each edge $(v_i,v_j)$. This condition ensures that if each $p_i$
is perturbed to
$p_i(t)=p_i+tu_i$, with $t$ small, then the edge length perturbations are of second order only
as $t$ tends to zero.
That is, for all edges,
\[
|p_i(t)-p_j(t)|-|p_i-p_j| = O(t^2).
\]

If $q(t):[-1,1]\to \H_v$ is a two-sided smooth flex of the finite framework
$(G,p)$ then $q'(0)$ is an infinitesimal flex
and for a generic finite framework every infinitesimal flex arises in this way.
 See Asimow and Roth \cite{asi-rot}
for example.

Associate with an infinite framework $(G,p)$ the  product vector space
$$\H_v = \prod_V \bR^d= \bR^d \oplus \bR^d \oplus \dots  $$
consisting of all sequences $u = (u_1, u_2, \dots )$. Conceptually such a vector
corresponds to a specification of instantaneous velocities, or to a  perturbation sequence,
 applied to the framework joints.
Define an \textit{infinitesimal flex} of $(G,p)$ to be a vector $u$ in $\H_v$ for which, as above,
 $\langle p_i-p_j,u_i-u_j\rangle =0$
holds for each edge $(v_i,v_j)$, and let $\H_{fl}$ denote the linear space of all these
vectors.
In the planar case $\H_{fl}$ contains the three-dimensional
linear subspace (assuming $G$ has at least one edge) of the infinitesimal flexes that arise from the isometries of $\bR^2$. Note that a nonzero rotation infinitesimal flex $u$ is an unbounded sequence if and only if $(G,p)$ is an unbounded framework.
We denote the space of rigid body motion infinitesimal flexes as
 $\H_{rig}$.

\begin{definition} An infinite framework $(G,p)$ is
{infinitesimally rigid} if every infinitesimal
flex is a rigid body motion infinitesimal flex.
\end{definition}

If $G$ is infinite then $\H_v$ contains properly
the direct sum space $\H_v^{00}= \sum_V \oplus \bR^2$
consisting of vectors whose coordinates are finitely supported, in the
sense of being finitely nonzero.
The following definition is convenient and evocative.

\begin{definition} An infinite framework $(G,p)$ is
{internally infinitesimally rigid} if every finitely supported infinitesimal flex
is the zero flex.
\end{definition}

\begin{center}
\begin{figure}[h]
\centering
\includegraphics[width=4cm]{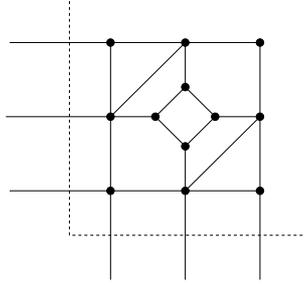}
\caption{Unit cell for an internally infinitesimally flexible periodic framework.}
\end{figure}
\end{center}

We now give the usual direct definition of the rigidity matrix $R(G,p)$ of a framework $(G,p)$,
allowing $G$ to be infinite. This matrix could also be introduced via
the Jacobian of the equation system that defines $V(G,p)$ since $2R(G,p)$
is the Jacobian  evaluated at $p$.

Write $p_i = (x_i, y_i)$, $u_i=(u^x_i, u^y_i), i=1,2, \dots $, and  denote the coordinate difference $x_i -x_j$ by $x_{ij}$. The rigidity matrix
is an infinite matrix $R(G,p)$ with rows indexed by edges $e_1, e_2, \dots$ and columns labeled by vertices but with multiplicity two, namely $v_1^x, v_1^y,v_2^x, v_2^y, \dots$. Note that any matrix of this shape,  with finitely many nonzero entries in each row, provides a linear transformation from $\H_v$ to $\H_e = \prod_E \bR$.

\begin{definition}
The rigidity matrix of the infinite framework $(G,p)$, with $p = (p_i) = (x_i,y_i)$,
is the matrix $R(G,p)$ with
entries $x_{ij}, x_{ji}, y_{ij}, y_{ji}$ occurring in the row
with label $e=(v_i, v_j)$ with the respective column labels
$v^x_i, v_j^x, v^y_i,v^y_j$, and with zero entries elsewhere.
\end{definition}

In particular a vector $u$ in $\H_v$ is an infinitesimal flex if and only if  $R(G,p)u=0$.

\begin{definition}
 (i) A  stress (or, more properly, a self-stress) of a finite or infinite framework $(G,p)$ is a vector $w = (w_e)$
in $ \H_e = \prod_E \bR$ such that  $R(G,p)^tw=0.$

(ii) A finite or infinite framework $(G,p)$ is isostatic, or absolutely isostatic,
if it is infinitesimally rigid and has no nonzero self-stresses.
\end{definition}

 Since it is understood here that $G$ is a locally finite graph a rigidity matrix hase finitely many entries in each column and so its transpose  corresponds
to a linear transformations from $\H_e$ to $\H_v$.

In the finite case a self-stress represents a finite linear dependence between
the rows of the rigidity matrix, which one might abbreviate, with language
abuse, by saying that the
corresponding edges of the framework are linearly
dependent.
A self-stress vector  $w= (w_e)_{e\in E}$ can be simply related to a vector $b=(b_e)$
conceived of as a sequence $b=(b_e)$ of bar tension forces with a resolution,
or balance, at each node. Indeed, for such a force vector
$b$  the vector $w$ for which $w_e=|p_i-p_j|^{-1}b_e$ ($e=(v_i,v_j)$) is a stress vector.
Thus there is a simple linear relationship between the space of internal stresses and the
space of resolving bar tensions.
We shall not consider here the more
general stress vectors, important in engineering applications, that arise from an external loading vectors.

Let $\H_e^{00}$ be the space of finitely supported vectors in $\H_e$. We say that
an infinite framework
$(G,p)$ is \textit{finitely isostatic} if it is internally infinitesimally  rigid and if
the finite support stress space $\H_{str}^{00}:=\H_{str}\cap \H_{e}^{00}
$ is equal to $\{0\}.$
It is straightforward to see that the grid frameworks $\G_{\bZ^d}$, in their ambient spaces, are finitely isostatic, as is the kagome framework.

Between the extremes of infinitesimal rigidity and internal rigidity there are
other natural forms of rigidity such as those given in the following definition.
Write $\ell^\infty $, $\ell^2$ and $c_0$  to indicate the usual Banach
sequence spaces for countable coordinates, and  write
$\H_e^\infty , \H_v^\infty, \dots , \H_v^0$
for the corresponding subspaces of $\H_e$ and $\H_v$.

\begin{definition}
An infinite framework $(G,p)$ is

(i) square-summably infinitesimally rigid (or infinitesimally $\ell^2$-rigid) if
$$\H _v^2 \cap \ker R(G,p) = \{0\},$$

(ii)  boundedly infinitesimally rigid (or infinitesimally $\ell^\infty $-rigid) if $$\H_v^\infty \cap \ker R(G,p) = \H_v^\infty \cap\H_{rig},$$

(iii)  vanishingly infinitesimally rigid (or infinitesially $c_0$-rigid) if
 $$\H_v^0 \cap \ker R(G,p) = \{0\},$$

(iv)   square-summably isostatic (or  $\ell^2$-isostatic) if it is infinitesimally $\ell^2$-rigid and $$\H_{str}\cap \H_e^2 =  \{0\},$$

(v)   boundedly isostatic if it is boundedly infinitesimally   rigid and $\H_{str}\cap \H_e^\infty =  \{0\}$,

(vi)   vanishingly isostatic if it is vanishingly infinitesimally rigid and $\H_{str}\cap \H_e^0 =  \{0\}$.
\end{definition}

There is companion terminology  for flexes and stresses. Thus
we refer to vectors in $\H_v^2 \cap \ker R(G,p)$ as
square summable infinitesimal flexes and so on.

\begin{example}{\rm
Let us use the shorthand $(\bN,p)$ to denote  a semi-infinite  framework in $\bR^2$ whose abstract graph is a tree with a single branch, with
edges $(v_1,v_2), (v_2,v_3), \dots $ and    where $p=(p_i), p_i=(x_i,y_i), i= 1,2,\dots$ .
Then,  writing $x_{ij}$ and $y_{ij}$ for the differences $x_i-x_j$ and $y_i=y_j$, as before, the rigidity matrix with respect to the natural ordered bases
takes the form
 \[
R(\bN,p) = \begin{bmatrix}
x_{12}&y_{12}&x_{21}&y_{21}&  0     &\dots       &   &\\
0     &0     &x_{23}&y_{23}& x_{32}& y_{32}& 0 & \dots\\
0 &0     &0    &0      &  *   &  *    &  * & \dots \\
\vdots &&&&&&&
\end{bmatrix}.
\]
With respect to the coordinate decomposition $\H_v = \H_x \oplus \H_y$
we have
\[
R(\bN,p) = \begin{bmatrix}
R_x &R_y\end{bmatrix} = \begin{bmatrix}
D_x &D_y\end{bmatrix}\begin{bmatrix}T& \\ &T\end{bmatrix}
\]
where $R_x = D_xT, R_y = D_yT$,  where
 $D_x$ and $D_y$ are the diagonal matrices
 \[
 D_x = \begin{bmatrix}
x_{12}&0&0& & \dots\\
0&x_{23}&0& &\dots\\
0&0&x_{34}&0&\\
\vdots &&&\ddots&
\end{bmatrix} \quad D_y = \begin{bmatrix}
y_{12}&0&0& & \dots\\
0&y_{23}&0& &\dots\\
0&0&y_{34}&0&\\
\vdots &&&\ddots&
\end{bmatrix}
 \]
and
$$T=\begin{bmatrix}
1&-1&0& & \dots\\
0&1& -1& &\dots\\
0&0&1&-1&\\
\vdots &&&\ddots&
\end{bmatrix}.
$$
If we now identify the domain and range spaces in the natural way for these coordinates
then
we have $T = I-U^t$ where $U^t$ is the transpose of the forward unilateral shift operator on the linear space of real sequences.

The analogous framework  $(\bZ,p)$ has a  similar matrix structure in all respects except that
in place of the Toeplitz matrix $T$ one has the corresponding
two-way infinite Laurent matrix
$I - W^{-1}$ where $W$ is the forward bilateral shift.
In both cases,
the two-dimensional subspace spanned by the
translation flexes is evident, being spanned by the constant vectors in $\H_x$
and $\H_y$.
Evidently there are infinitely many finitely supported flexes
and in fact it is possible to identify $\ker R(G,p)$ as a direct product space.

One can use operator formalism
to examine the space of stresses.
In the case
of the simple framework $(\bZ,p)$
note that $W^t=W^{-1}$ and
$$ \ker (I-W) = \ker (I-W^{-1}) =  \bR e$$ where
$e$ is the vector with every entry equal to $1$. Since
\[
R(\bZ,p)^t = \begin{bmatrix}I-W& \\ &I-W\end{bmatrix}\begin{bmatrix}D_x \\D_y\end{bmatrix}
\]
it follows that a vector $w$ is a stress vector if and only if $D_xw \in \bR e$ and $D_yw \in \bR e$.
Thus for some constants $\alpha, \beta$ we have
$
x_{i,i+1}w_i = \alpha, y_{i,i+1}w_i = \beta $, and so
for all $i$
\[
\frac{y_i-y_{i+1}}{x_i-x_{i+1}}=\frac{\beta}{\alpha}.
\]
This colinearity condition shows that the space of stresses is trivial  unless the framework points $p_i$, $i \in \bZ$ are colinear in which case $\H_{str}$ is one dimensional. This includes the colinear cases in which $p$ is a bounded sequence and the framework lies in a finite line segment in $\bR^2$.
}
\end{example}

\begin{example}
{\rm With similar notational economy write
 $(\bZ^r,p)$, (resp. $(\bN^r,p)$) for frameworks associated
with the  \textit{grid graph}
 with vertex set labeled by $r$-tuples of integers (resp. positive integers)
 $n=(n_1,\dots ,n_r)$  where the edges correspond to vertex pairs
 $(n, n \pm e_j)$, where $e_1, \dots , e_r$ are the usual basis elements.
 The ambient space for the framework is either understood or revealed by the entries of the vector $p$.

Again one can use operator formalism to analyse the space of stresses
as a vector subspace of $\H$.
In the special case of the regular grid framework $\G_{\bZ^2}$ in $\bR^2$ one can  see that the vector subspace $\H_{str}$  is a direct
product vector space (like $\H$ itself) whose product basis is
indexed by (two-way infinite)
linear subframeworks parallel to the coordinate axes.
This is also true for a general \textit{orthogonal grid framework} such as the bounded
grid framework determined by the framework points $(\pm (1-(1/2)^i), \pm (1-(1/2)^j)).$}
\end{example}

Note that we have defined an infinitesimal flex in a local way, being
the verbatim counterpart
of the usage for finite frameworks. In particular the notion takes no account of the possibility of (second order) amplification or vanishing flexibility.

{\rm
Our examples above indicate  the importance of
shift operators and in the next section we see that the bilateral shift operators, in their Fourier transform realisation as multiplication operators, play a central role in the discussion of periodic frameworks.

Let us note here that approximate infinitesimal flexes are natural for infinite frameworks and  the operator theoretic perspective.

\begin{definition}
An approximate square-summable flex of an infinite framework $(G,p)$
is a sequence of finitely supported unit vectors $u_1, u_2, \dots $ in $\H_v^2$ (or $\K_v^2$)
such that $\|R(G,p)u_n\|_2 \to 0$ as $n\to \infty$.
\end{definition}

{ Let $(G,p)$ be a distance regular framework. Then it is straightforward to show that
the rigidity matrix determines a bounded Hilbert space operator $R$. It is the metrical and geometric properties
of the action of $R$ and its transpose that
have relevance to rigidity theory rather than the spectral theory of $R$.
However we do have the  almost vacuous statement
that the existence of approximate (square-summable) flexes corresponds to the point $0$ belonging to the approximate point spectrum of $R$. For if $0$ lies in the approximate point spectrum then (by definition) $Rv_n$ is a null sequence for some sequence of unit vectors $v_n$, and approximation of these unit vectors by vectors with finite support yields, after normalisations, an approximate square-summable flex sequence $(u_n)$.
}

It is implicit in the matricial function association below that the rigid unit modes of translationally periodic frameworks are tied to
the existence of approximate flexes.
}

\section{Crystal Frameworks and Flexibility}
In previous sections we have constructed  frameworks  to illustrate
various definitions and properties.
It is perhaps of wider interest to understand, on the other hand, how extant infinite frameworks, such as those suggested by crystals or repetitive
structures, may be flexible.
Accordingly we now define crystal frameworks
and investigate various forms of flexibility and rigidity.

\subsection{Periodic and crystal frameworks.}
 We have already observed some properties of the basic examples of the \textit{grid framework} $\G_{\bZ^2}$, the \textit{squares framework} $\G_{sq}$ and the kagome framework, $\G_{kag}$.
In $\bR^3$ we also have analogues, such as the cube framework
$\G_{cube}$,  the octagon framework $\G_{oct}$ and the kagome net framework $\G_{knet}$, which consist, respectively, of vertex-joined cubes, octahedra and tetrahedra, with no shared edges or faces, each in a natural periodic arrangement.
The frameworks $\G_{kag}$,
$\G_{cube}, \G_{oct}$ and $\G_{knet}$ are polytope body-pin frameworks but we  consider  the polytope rigid units as bar-joint subframeworks
formed by adding some, or perhaps all, internal edges. As such these frameworks  are examples of crystal frameworks in the  sense
of the formal definition below. We first comment on a wider notion of periodicity.

 \begin{definition} An {affinely periodic framework} in $\bR^d$ is a framework $\G=(G,p)$ for which there exists a non-trivial discrete group  of affine transformations $T_g, g \in \D$, where each $T_g$ acts on framework points and framework edges.
 \end{definition}


For example, the  two-way infinite dyadic cobweb framework of Figure 15
is affinely periodic for the dilation doubling map and the four-fold dihedral group $D_4$.

\begin{center}
\begin{figure}[h]\label{f:kagstripflex}
\centering
\includegraphics[width=5cm]{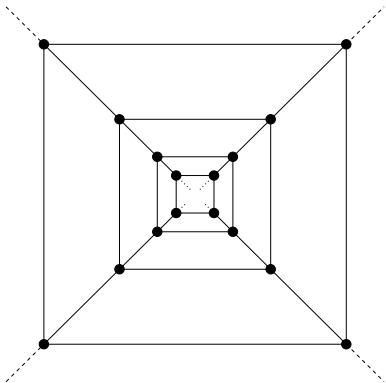}
\caption{An affinely periodic cobweb framework.}
\end{figure}
\end{center}

For another example we may take the infinite $\bZ$-periodic framework
in three dimensions for which
 Figure 16 forms a perspective view down a central axis, with framework vertices
 $(\pm 1, \pm 1, m), m\in \bZ$.
Here the affine group is an isometry group isomorphic, as a group, to $\bZ \times C_2 \times D_4$.

To illustrate the following definition observe in Figure 16 a template of six edges
and three vertices which generates the kagome framework by the translations
associated with the parallelogram unit cell. Borrowing crystallographic terminology
we refer to such a template as a {\it motif}
for the framework and the chosen translation group. The following formal definition gives a convenient
way of specifying abstract crystal frameworks.

\begin{center}
\begin{figure}[h]
\centering
\includegraphics[width=9cm]{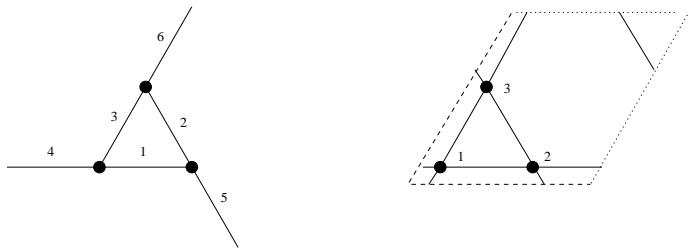}
\caption{A motif and unit cell for the kagome framework.}
\end{figure}
\end{center}

\begin{definition}
A {crystal framework} $\C = (G,p)$ in $\bR^d$ is a connected bar-joint framework for which there is a discrete
group  of translation isometries $\T = \{T_g: g \in \D\}$, a finite connected set $F_e$ of frameworks edges and a finite set $F_v$ of framework vertices (being a subset of the vertex set of  $F_e$) such that

(i) the unions
\[
 \cup_{g\in \D} T_g(F_e), \quad \quad  \cup_{g\in \D} T_g(F_v)
\]
coincide with the sets of framework edges and vertices, and

(ii) these unions are unions of  disjoint sets.
\end{definition}

 We denote a crystal framework $\C$ by the triple $(F_v,F_e, \T)$ or by
 the triple $(F_v,F_e, \bZ^d)$
 in the case of integer translations.
An associated \textit{unit cell}
for $\C$ may be defined as a
set which
contains $F_v$ and no other framework
points and for which the  translates under the translation isometries are disjoint and partition the ambient space.
For example in the case of $\G_{\bZ^2}$ we may take the semiopen set  $[0,1)^2$
or the set $[0,1)\times [1/2,5/2)$ as unit cells. Such parallelepiped unit cells are useful for us for  torus models for crystal frameworks.
Voronoi cells (Brillouin zones) also play a unit cell role
in applications but we shall not need such geometric detail here.

In many applied settings the appropriate framework models
have "short"  edges,
spanning no more than two adjacent unit cells. Here we allow general edges which may span a chain of adjacent cells.

Recall from elementary crystallography that, modulo orthogonal transformations,
there are 14 different forms (or symmetry types) in which a countable set of isolated points can be arranged with translational symmetry throughout three-dimensional space. These arrangements are called the Bravais lattices and the translation group $\T$ above corresponds to such a lattice.
Thus each point of the framework lies in the Bravais lattice generated by the orbit of its unique corresponding motif vertex
under the translational group.

\subsection{Deformability and flow flexibility.}
Recall that a general (countable, locally finite, connected) framework $\G$
is \textit{rigid} if there is no base-fixed continuous
flex and is \textit{boundedly rigid}, or boundedly nondeformable, if there is no
bounded base-fixed continuous flex $p(t)$. Recall, from Section 2, that bounded flexes are those for which there is an absolute constant $M$ such that
for every vertex $v$ the time separation $|p_v(t) - p_v(0)|$ is bounded by $M$
for all $t$ and all $v$.

We first describe a context for the standard "alternation" flexes of $\G_{sq}$ and $\G_{kag}$
and certain  periodic flexes of $\G_{kag}$ with reduced symmetry. For these nonbounded flexes
translational periodicity is maintained but relative to an affine flow of  the ambient space. By an \textit{affine flow} we mean simply a path $t\to A_t$ of affine transformations of $\bR^d$ which is pointwise continuous. The simplest such flow in
two dimensions is
a contracting flow such as $A_t(x,y) = ((1-tc)x,(1-tc))y$,  $0<c<1$.  The alternating flexes of $\G_{sq}$  and $\G_{kag}$ are associated with such a flow.

\begin{definition}
Let $d=2,3$, let $t \to A_t$ be a flow of $\bR^d$ and let $\C = (G,p)$ be a crystal framework
for the translation group $\T = \{T_g: g \in \D\}$. A flow-periodic flex of $\C$, relative
to the flow and the translation group $\T$, is a continuous flex $p(t)$ such that for each $t$ the framework $(G,p(t))$ is
$\T^t$- periodic, where $\T^t = \{A_tT_gA_t^{-1}: g \in \D\}$.
\end{definition}

Let us say simply that $\C$ is {\it affinely periodically deformable} if there exists a non trivial flow-periodic flex.

A flow-periodic flex $p(t)$ for $\C$ can be defined for a given flow if and only if
for each $t$ one can continuously solve the distance constraint equations for the vertex
positions of the motif, with the periodicity constraint, in the $A_t$-deformed unit cell. Equivalently, the motif and unit cell define a finite framework on a  torus with the noninterior edges of the motif providing reentrant ("locally geodesic") edges on the torus. For example consider the torus framework in Figure 17.
\begin{center}
\begin{figure}[h]
\centering
\includegraphics[width=3.5cm]{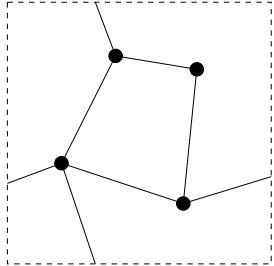}
\caption{A torus framework.}
\end{figure}
\end{center}
A horizontal affine contraction
$H_t : (x,y)\to ((1-t)x,y)$ leads to a continuous flex of the torus framework and hence to a colossal flex of
the associate crystal framework. The same is true for the vertical affine contraction $V_t$ and for the skew affine transformation
$$S_t : (x,y) \to (x+(\sin t)y,y+ (1-\cos t)x).$$ Note that this particular transformation
preserves both the cyclic width and cyclic height of the torus.

For an illustration of the method we note the following affine deformation result for what might be termed
periodic cell-generic grid frameworks in $\bR^d$.
For more general results of this nature see also Borcea and Streinu \cite{bor-str}.


\begin{thm}
Let $\C= (\bZ^d, p)$ be a grid framework in $\bR^d$
which is $(n_1,\dots , n_d)$-periodic and which is a $1/3$-perturbation of $\G_{\bZ^d}$
in the sense that
$$|p(k_1,\dots ,k_d) - (k_1,\dots ,k_d)|<1/3$$
for $0\leq k_i \leq n_i-1, 1\leq i\leq d$.
Then $\C$ is affinely periodically deformable.
\end{thm}

\begin{proof}
We sketch the proof in case $d=2$ with $(n_1,n_2) = (n,m)$. For convenience re-scale the framework so that the large cell $[0,n-1)\times [0,m-1)$  of $\C$ becomes
to the usual unit cell.
Let  $(F_v, F_e, \bZ^2)$ be the motif representation
of $(\bZ^2, p)$.
Suppose first that $n=m=2$ and the motif has four noninterior edges
(corresponding to the four reentrant edges on the associated torus).
Let $F_e'$ be $F_e$ with one vertical and one horizontal reentrant edge removed
and let $\C'$ be the associated framework. (See Figure 17.) For $t$ taking
positive  values in some finite interval  there are affine deformations
of $\C'$ associated with each of the flows $H_t, V_t, S_t$. Moreover there is an
affine deformation associated with any composition  $A_t= V_{\beta(t)}H_{\alpha(t)}S_t$, where the functions $\alpha$ and $\beta$ are any continuous functions with $\alpha(0) = \beta(0) =0$, where  $t$ takes values in some finite interval. Now note that we may chose $\alpha(t)$ so that the separation distance corresponding to the omitted horizontal edge is constant, and we may choose $\beta(t)$ similarly for the omitted vertical edge. Note also  the essential fact that   $V_{\beta(t)}$ does not change the cyclic width.
In this way an affine deformation for $\C$ is determined. (In the case of generic
points, is determined uniquely on some finite interval).

The same principle operates for general $(n,m)$. The  subframework $\C'$, with two deleted noninterior edges (one for each coordinate direction)   once again has a flow periodic deformation associated with a composition  $A_t= V_{\beta(t)}H_{\alpha(t)}S_t$, for $t$ in a sufficiently small interval. This follows from a simple induction argument.
Once again the functions $\alpha$ and $\beta$ are chosen to provide a deformation of $\C'$ in which the separations for the omitted edges is constant.
In this way a flow periodic deformation of $\C$ itself is determined.
\end{proof}

The deformations obtained in the last proposition  are colossal deformations
and this seems to be a necessary condition if some form of periodicity
is to be maintained.
It would be interesting to determine when such frameworks possess bounded
"unstructured"  deformations.

One may also consider affine deformations relative to a subgroup of $\D$
associated  with a  supercell. By a \textit{supercell} we mean a
finite union (or possibly an infinite union) of adjacent cells which tile the ambient space by translations from the subgroup.
In this case the deformations
maintain  only a longer period of translational symmetry and for an infinite linear supercell  one may even forgo translational symmetry
in one direction.

To illustrate this let us note a class of
interesting deformations of the kagome framework $(G,p)$ which have this form.

Start with an alternation flex $p(t)$ which
leaves fixed a particular vertex $p_*$ and leaves invariant a line of hexagon diameters.
Note that this flex (which is not base-fixed) has bilateral symmetry with respect to
the line (viewed as a mirror line) but that the inversion symmetry of $(G,p)$  about the fixed vertex is broken for $t>0$.
Perform identical surgeries on the frameworks $(G,p(t))$ as follows:

(i) cut the frameworks  along the fixed  line,

(ii) effect a reflection of one of the resulting half-planes frameworks in
the orthogonal line through $p_*$ and
rejoin the half-planes to create a new flex of $(G,p(0))$.

We might view this flex as one with a symmetry transition line. The hexagons in this line maintain inversion symmetry  while elsewhere the hexagons maintain bilateral symmetry.

One may perform such surgery  on several
parallel surgery lines simultaneously. Performing surgery on countably many
lines leads to the following theorem.
 This shows, roughly speaking, that  the kagome framework sits in its configuration space as an infinitely singular point in the sense that it is the starting point for uncountably many distinct flexes. (Here we use the term distinct in the sense of
Definition 2.7.)

\begin{thm}
There are uncountably many distinct flow periodic flexes of the kagome framework.
\end{thm}


\subsection{Crystal frameworks and periodic infinitesimal flexibility.}
Let $\C= (G,p)= (F_v, F_e, \bZ^2)$ be a crystal framework in the plane  for the integer translation isometry group. A natural form of infinitesimal flex for $\C$
is that of a $1$-cell periodic flex in the following sense.

\begin{definition}
A vector $u= (u_v)_{v \in V}$ in the real vector space $\H_v$, or in the complex vector space $\K_v$, is a $1$-cell-periodic infinitesimal flex
for $\C$ (or simply a periodic flex if there is no ambiguity)  if $u \in \ker R(G,p)$
and $u_{\kappa +n} = u_\kappa$ for all $\kappa \in F_v$ and $n \in \bZ^d$.
\end{definition}

Such an infinitesimal flex $u$ is a bounded sequence of vectors  determined by
periodic extension of what we may call the \textit{motif flex vector}
$u_{motif} = (u_\kappa)_{\kappa\in F_v}.$
There is  a one-to-one correspondence between
these periodic flexes and the finite vectors that are in the kernel of
the \textit{motif rigidity matrix} $R_m(G,p)$ which we  may   define as the
natural "periodic completion"
of the $|F_e| \times d|F_v|$ submatrix of $R(G,p)$.
This is the natural representing matrix for $R(G,p)$ viewed as a linear
transformation between the finite dimensional subspaces of $\bZ^2$-periodic vectors.
If there are no reflexive edges in the motif then

Similarly we define periodic infinitesimal stresses as those which correspond to periodic extensions of vectors
in the cokernel of the motif rigidity matrix.

The crystal framework
$\C$, with given motif and discrete translation group indexed by $\D$, is said
to be \textit{$1$-cell-periodically isostatic} (or simply periodically isostatic) if the only $1$-cell periodic flexes are translation flexes and if there are no nontrivial $1$-cell periodic stresses.

Once again it can be helpful to consider a flat torus model for such frameworks.
Note for example that the generic periodic framework defined by Figure 17 is
periodically rigid and indeed periodically isostatic.

The following interesting periodic variant of Laman's theorem has been obtained recently by E. Ross  \cite{ros}.
Let us say that the planar periodic framework $\C$ is topologically proper if
$\C$ is connected and for every  torus subframework motif $F_v', F_e'$ with
$2|F_v'| - |F_e'| = 2$ there is an
edge cycle from $F_e'$ which properly wraps around the torus in the sense that the associated homotopy class is nonzero.

\begin{thm}
Let $\C= (F_v, F_e, \bZ^2)$ be an infinite framework in the plane which is periodic for the integer translation group and is topologically proper.
Then following are equivalent.

(i) $
2|F_v| - |F_e| = 2
$
and for all edge induced submotifs $F_v', F_e'$ we have
$2|F_v'| - |F_e'| \geq 2$.

(ii) $\C$ is periodically isostatic in the sense that the periodic vectors in $\ker R(G,p)$ are spanned by the two translations (periodic rigidity) and the periodic vectors in
$\coker R(G,p)$ are zero (periodic stress free).
\end{thm}

As we have seen, a periodically infinitesimally rigid framework in the sense above may nevertheless
correspond to a flat torus framework which can flex infinitesimally
(and even deform)
relative to contraction or expansion of the unit cell. Allowing such vertical and horizontal freedoms adds two further degrees of freedom to the constraint
equations. It would be of interest to obtain a similar characterisation
in this case as well as for the triply flexible affine torus.


\section{The Matricial Symbol Function of a Crystal Framework}
We now derive  matrix function operators for general (abstract) crystal frameworks.
Our approach is decidedly Hilbert space theoretic and allows for an extended conceptual
framework for rigidity analysis.
First we outline the standard identification of operators commuting with shift operators and multiplication
operators in a Fourier transform space. We  adopt complex scalars, replacing the $\H_v$ spaces of Section 3 by their complex counterparts,  $\K_v, \K_e,
\K_v^\infty, \K_v^2, $ and so forth.

\subsection{Matrix function multiplication operators.}
Let $\phi $ be a continuous complex-valued function on the unit circle $\bT$. It defines a multiplication operator $T$ on the usual complex Hilbert space $L^2(\bT)$. Its representing matrix with respect to the standard orthonormal basis $\{z^n: n \in \bZ\}$ is the $\bZ \times \bZ$ indexed matrix with entries
\[
T_{i,j} = \langle\phi z^j,z^i\rangle =  \langle\phi z^{j-i}, 1\rangle
= \hat{\phi}(i-j),
\]
the $(i-j)^{th}$ Fourier coefficient of $\phi$.

Similarly, let $\Phi(z)$ be a continuous matrix-valued function on the two or three dimensional
 torus $\bT^d$ taking values in the space of
$n \times m$ complex matrices. One can specify such a function $\Phi(z)= \Phi(z_1,\dots ,z_d)$
 in terms of a matrix of scalar
functions $\Phi(z) = [\phi_{k,l}(z)]_{k=1,l=1}^{n,m}$. Given in this way $\Phi(z)$   defines a
multiplication operator from $L^2(\bT)\otimes \bC^m$ to $L^2(\bT)\otimes \bC^n$.
Indeed, choose a basis $\{\xi_l\}$ for $\bC^m$  and an associated basis
$\{\xi_l \otimes z^p\}$ for the domain space. Similarly let  $\{\eta_{k} \otimes z^p\}$
be a basis for the codomain.
Then the operator $T$ of multiplication by $\Phi(z)$ may be defined by specifying, for each continuous function $f(z)$ on $\bT^d$,
\[
T(\xi_l \otimes f(z))= \sum_{k=1}^m \eta_{k} \otimes \phi_{k,l}(z)f(z),
\]
and extending by linearity and continuity.
Thus we can see that the representing matrix for $T$ with respect to these bases
is determined by the Fourier coefficients of the matrix
entries for $\Phi$:
\[
 \langle T(\xi_l \otimes z^p), \eta_k \otimes z^q\rangle_{\bC^m\otimes L^2(\bT^d)}
 =  \langle\phi_{k,l}(z)z^{p-q}, 1\rangle_{L^2(\bT^d)} = \hat{\phi}_{k,l}(q-p).
\]
Viewing an element of $L^2(\bT^d)\otimes \bC^m$ as a function $F(z)$ taking values in
$\bC^m$ (strictly speaking, taking values almost everywhere), and similarly for vectors in the codomain, one also considers $TF$ as the function $\Phi(z)F(z)$. This operator is usually denoted
as $M_\Phi$.

Suppose now,
starting afresh, we have  orthonormal bases
\[
\{\xi_{l,p}: 1\leq l \leq m, p \in \bZ^d\}\quad \{\eta_{k,p}: 1\leq k \leq n, p \in \bZ^d\}
\]
for Hilbert spaces such as $\K_v^2$ and $\K_e^2$ respectively.
Suppose also
 that we are given  an operator $T$ from  $\K_v^2$ to $\K_e^2$ by means of its representing matrix with respect to these bases
 and suppose moreover
that this matrix has the translational symmetry above in the sense that the matrix entries
\[
\langle T\xi_{l,p}, \eta_{k,q}\rangle_{\K_e^2}
\]
are independent of $p-q$.  Let $\F_v: \bC^n\otimes L^2(\bT^d)\to \K^2_v$  be the Fourier transform and $\F_e$ the Fourier transform for $\K^2_e$.
These are simply the unitary operators determined by the natural bijection of basis elements,
namely
$\xi_{l,p}\to \xi_{l}\otimes z^p$ and $\eta_{k,p}\to \eta_{k}\otimes z^p$, respectively.
Thus, from the scalar matrix entries the functions
$\phi_{k,l}$ are determined and hence the matrix function $\Phi(z)$. In  this way one identifies $\F_e^{-1}T\F_v$ as the associated multiplication operator $\M_\Phi$.

\subsection{The rigidity matrix as a multiplication operator}
Let us  keep in view  the special case
of the $\bZ^2$-periodic grid framework determined by an algebraically generic quadrilateral
in the unit cell.

Let $p^1,\dots ,p^4$ be four framework vertices in the unit cell
$[0,1)^2$ constituting the motif set $F_v$ of $\bZ^2$-periodic
quadrilateral grid $(\bZ^2,p)$. Let $p^t_{ij} = p^t+ (i,j)$,  for $(i,j) \in \bZ^2$, be the general framework points and
let $e^1, \dots , e^8$ denote the eight edges which form the motif set $F_e$ given by
\[
e^1 = [p^1,p^2], e^2 = [p^2,p^3], e^3 = [p^3,p^4], e^4 = [p^4,p^1],
\]
\[
e^5 = [p^2_{-1,0},p^1], e^6 = [p^3_{0,-1},p^2], e^7 = [p^4_{1,0},p^3], e^8 = [p^1_{0,1},p^4],
\]
and let $e^t_{i,j} = e^t +(i,j)$ be a typical edge.
Write the corresponding basis for the  vector space $\K^0_e$ as
$$\{\eta_{t,i,j}: 1\leq t\leq 8, (i,j) \in \bZ^2\}.$$ Also, label the natural basis for $\K^0_v$
as the set of vectors $\{\xi^x_{s,i,j}, \xi^y_{s,i,j} \}$, where $s$ ranges from $1$ to  $4$ and $(i,j)$ range through $ \bZ^2$.

The rigidity matrix $R(G,p)$
satisfies the following symmetry equations,
 $$
 W_{i,j}R(G,p)U_{i,j}= R(G,p)
 $$
 with respect to
the translation shift operators $U_{i,j}$ on $\K^{00}_v$ and $W_{i,j}$ on $\K^{00}_v$.
By the discussion above the rigidity matrix for this framework gives rise to an
$8 \times 8$ matrix function $\Phi(z_1,z_2) = (\phi_{i,j}(z_1, z_2)$ on the torus $\bT^2$ in $\bC^2$
which is determined by the equations
\[
\hat{\phi}_{k,\{x,s\}}(-(i,j))= \langle(R(G,p)\xi^x_{s,i,j}, \eta_{k,0,0} \rangle,
\]
together with a companion set of equations  for the $y$-labeled basis elements.
Furthermore there is a simple algorithm for computing the symbol matrix function
from the  motif $(F_v, F_e)$ and the motif rigidity matrix.

For a general crystal frameworks in the plane
we have the following recipe for identifying $\Phi(z_1,z_2)$.
There is an entirely similar identification of the matrix function of a crystal framework in
higher dimensions.

\begin{thm}{Let $\Phi(z)$ be the matricial symbol function of the crystal framework $\C = (F_v, F_e, \bZ^2)$. The entry of the rigidity matrix determined by the edge $e_k$ of $F_e$ and the column labeled $v^x_{s,i,j}$ (resp. $v^y_{s,i,j}$) provides the $(-i,-j)$-th Fourier coefficient of ${\phi}_{k,\{x,s\}}$ (resp. ${\phi}_{k,\{y,s\}}$).}
\end{thm}

Note that for the motif for the periodic quadrilateral
grid each edge has vertices with different $s$ index. This is rather typical
and in such cases it follows that ${\phi}_{k,\{x,l\}}$ is either zero or has one nonzero
Fourier coefficient.
In the case of a "reflexive" edge in $F_e$, of the form $[p_{\kappa,0}, p_{\kappa,\delta}]$
there are entries in the columns for $\kappa$ corresponding to
$(1-\ol{z}^\delta)v_e$ where $v_e$ is the usual vector of coordinate differences appearing in the rigidity matrix for the edge $e$. That is, $v_e= p_{\kappa,0}- p_{\kappa,\delta}$.

Thus we may obtain the following identification where
the direct construction of $\Phi$ from $\C$ is given above.

\begin{thm}
Let $\C =(F_v, F_e, \bZ^d)$ be a crystal framework with $m=|F_e|$,
$n=|F_v|$ and with rigidity operator $R$ from $\K_v^2$ to $\K_e^2$.

(i) The motif
$(F_v, F_e)$ determines a
matrix-valued trigonometric function
$\Phi : \bT^d \to M_{m,dn}(\bC)$
for which there is a
 unitary equivalence
$\F_e^{-1}R\F_v = M_\Phi $
where $M_\Phi$ is
the multiplication operator
\[
M_\Phi :  \bC^{nd} \otimes L^2{(\bT^d)} \to \bC^{m}  \otimes L^2(\bT^d).
\]

(ii) The framework $\C$ is square-summably rigid (resp. square-summably stress-free)
if the column rank (resp. row rank)
of $\Phi(z)$ is maximal for almost every $z \in \bT^d$.
\end{thm}

The following corollary
shows that the existence of a square summable flex
is  a rather strong condition for a crystal framework and in fact "typical" frameworks in Maxwell counting equilibrium have no such flexes.
What is of particular significance for such frameworks however is the presence
of approximate flexes, in the sense noted in Section 3, as these are related to the
local rank degeneracies of $\Phi(z)$ and to periodic flexes.

\begin{cor}
The following are equivalent for a crystal framework with Maxwell counting equilibrium.

(i) $\C$ has an  nonzero internal (finitely supported)  infinitesimal flex.

(ii) $\C$ has an nonzero summable infinitesimal flex.

(iii)  $\C$ has an nonzero square-summable infinitesimal flex.

\end{cor}

\begin{proof} Note that
(i) implies (ii) and  (ii) implies (iii) so it remains to show (iii) implies (i).

Suppose that the matrix $\Phi(z)$  has a square-summable
vector $f(z)$ in its kernel. Then $\det(\Phi(z))=0$ for almost all points in the
support of $f(z)$. Since $\det\Phi(z)$ is a mutivariable trigonometric polynomial
it necessarily vanished identically, that is the polynomial $\det\Phi(z)$ is the zero
polynomial.

Recall that for any square matrix $X$ over a ring we have \\
$X\tilde{X}=\det(X)I_n$
where $\tilde{X}$ is given by the usual formula \\
$\tilde{X}_{i,j}= (-1)^{i+j}X_{i,j}\det (X^{ij})$ involving the cofactors $(X^{ij})$.
In particular if $\tilde{X}$ is not zero and $\det(X)=0$ then for a nonzero column
vector $f(z)$ of $\tilde{X}$ we have $Xf(z)=0$. Applying this to $X=\Phi$ provides the desired polynomial (finitely supported) vector in this case.

If $\tilde{X}$ happens to be the zero matrix we make use of the minimal polynomial lemma
below.
To apply the lemma let $q$ be the minimum polynomial of $X$
so that $q(X)=0$.
By the lemma we have $q(X) = Xq_1(X)$ and $q_1(X) \neq 0$ by minimality. Thus there is a
nonzero vector $f(z)$ in the range of $q_1(X)$ and this provides the desired local flex.
\end{proof}

\begin{lem} Let $q(\lambda)$ in $R[\lambda]$ be the minimal polynomial of an
$n$ by $n$ matrix $X$ whose entries lie in an integral domain $R$ and
suppose that $\det X=0$. Then $q(0)=0$.
\end{lem}

\begin{proof} Let $p(\lambda)$ be the characteristic polynomial $\det (\lambda I_n-X)$.
Recall from the Cayley-Hamilton theorem that $p(X)=0$.
In the algebraic closure of the field of
fractions of $R$ the linear factors of $q(\lambda)$ and $p(\lambda)$ agree.
See Lange \cite{lan-alg} for example. Thus
\[
q(\lambda) = \prod (\lambda-\alpha_i)^{s_i}, \quad p(\lambda) = \prod (\lambda-\alpha_i)^{r_i}
\]
with $s_i \leq r_i$ for each $i$. In particular for suitably large $n$
the polynomial $q(0)^n$ is divisible by $p(0)$. Since $\det X=0$ it follows that $p(0)=0$
and so $q(0)^n=0$. Since $q(0)$ is in $R$ it is equal to zero and the proof is complete.
\end{proof}

\subsection{From motifs to matrix functions.} We now consider some examples.
Adjusting notation, write $z,w$ for the coordinates variables of $\bT^2$ so that a trigonometric
polynomial takes the form a finite sum
\[
\phi(z,w) = \sum_{i,j} \hat{\phi}(i,j)z^iw^j.
\]
Examining the recipe of Theorem 5.1 above in the case of the periodic quadrilateral
grid we see that the internal edges
$e^1, \dots ,e^4$ of the motif provide
four rows for $\Phi(z,w)$
whose entries are constant functions
with constants corresponding to those of the rows for the rigidity
matrix $R(G,p)$. The other edges
provide rows according to
the following simplified monomial rule. As we have intimated above, this rule applies generally
when all the "external" edges of the motif have vertices with distinct $s$ index:
\bigskip

\textit{ the entries of the $e^{th}$ row of $\Phi(z,w)$ that appear in the columns
for the external (non-motif) vertex  are
the corresponding entries for the motif rigidity matrix multiplied by the
monomial $\ol{z}^i\ol{w}^j$, where $(i,j)$ is the shift index for the cell occupied
by the external vertex.}
\bigskip

Thus we obtain the matrix identification in the next proposition.
We write $p_i = (x_i, y_i)$, for $ i=1,\dots ,4$.
Note, for example, that the entry $x_{12}+1 = x_1-(x_2-1)$ denotes the constant function
corresponding to the entry of $R(G,p)$
for row $e^5$ and column $x_1$, while $-(x_{12}+1){z}$ is the entry
corresponding to the $x$-coordinate of the external vertex of $e^5$.

\begin{prop}
The matricial symbol function $\Phi(z,w)$ of the periodic quadrilateral grid
framework $(\bZ^2,p)$ for the four vertex motif
is the  matrix function on $\bT^2$ given by
{\small
$$ \left[ \begin {array}{cccccccc} { x_{12}}&{ y_{12}}&-{ x_{12}}&-{
y_{12}}&0&0&0&0\\\noalign{\medskip}0&0&{ x_{23}}&{ y_{23}}&-{ x_{23}}&-{
 y_{23}}&0&0\\\noalign{\medskip}0&0&0&0&{ x_{34}}&{ y_{34}}&-{ x_{34}}
&-{ y_{34}}\\\noalign{\medskip}-{ x_{41}}&-{ y_{41}}&0&0&0&0&{ x_{41}}
&{ y_{41}}\\\noalign{\medskip}{ x_{12}}+1&{ y_{12}}&-{
(x_{12}}+1){z}&-{ y_{12}}{z}&0&0&0&0\\\noalign{\medskip}0&0&{ x_{23}}&{ y_{23}}
+1&-{ x_{23}}{w}&-{ (y_{23}}+1){w}&0&0\\\noalign{\medskip}0&0&0&0
&{ x_{34}}-1&{ y_{34}}&(1-{x_{34}})\ol{z}&-{ y_{34}}\ol{z}\\\noalign{\medskip}-{ x_{41}}\ol{w}&(1-{ y_{41}})\ol{w}&0&0&0&0&{ x_{41}}&{ y_{41}}-1\end {array} \right]
$$
}
Furthermore, if the four points in the unit cell have algebraically independent coordinates then $(\bZ^2,p)$ is square-summably isostatic.
\end{prop}

\begin{proof} The first four rows are linearly independent since
the points $p_1,\dots ,p_4$ have algebraically independent coordinates.
It can be shown from elementary linear algebra
for almost every $z,w$ the full matrix has rank $8$.
Thus  Theorem 5.2 applies.
\end{proof}

As a special case we obtain the matrix symbol function for $\G_{\bZ^2}$
determined by the eight-edged motif, namely
{\small
$$ \Phi(z,w) = 1/2\left[ \begin {array}{cccccccc} -1&0&1&0&0&0&0&0
\\\noalign{\medskip}0&0&0&-1&0&1&0&0\\\noalign{\medskip}0&0&0&0&1&0&-1&0\\\noalign{\medskip}0&-1&0&0&0&0&0&1\\\noalign{\medskip}
1&0&-{z}&0&0&0&0&0\\\noalign{\medskip}0&0&0&1&0&-{w}&0&0
\\\noalign{\medskip}0&0&0&0&-1&0&\ol{z}&0\\\noalign{\medskip}0
&\ol{w}&0&0&0&0&0&-1\end {array} \right]
$$
}
The determinant in this case is
\[
-{\frac {1}{256}}\,zw{{ \left( 1-\ol{z} \right)^2  \left( 1-\ol{w} \right)^2 }}.
\]

On the other hand for this symmetric quadrilateral framework $\G_{\bZ^2}$ the formalism above
can be applied to a smaller unit cell with the simple two-edged motif
of Figure 18.
\begin{center}
\begin{figure}[h]
\centering
\includegraphics[width=6cm]{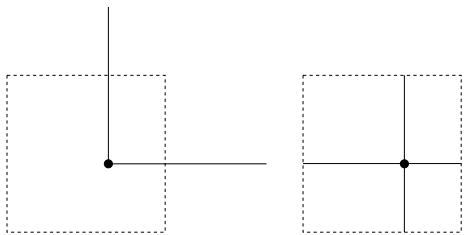}
\caption{A motif and unit cell for $\G_{\bZ^2}$.}
\end{figure}
\end{center}
Note  that the index $s$ for the recipe
takes the value $1$ only and  both motif edges have vertices with the same
$s$ index. It follows that the rigidity operator of $\G_{\bZ^2}$ associated with the motif  is unitarily equivalent to the multiplication operator on the Hilbert space
$ \bC^2 \otimes L^2(\bT^2) $
determined by the matricial symbol  function
$$\Psi(z,w) =\begin{bmatrix}
\ol{z}-1&0\\
0& \ol{w}-1
\end{bmatrix}.
$$

One can similarly verify the following.

\begin{prop}(i)
The matricial symbol function of the kagome framework
determined by the six-edge motif is
given by
\[
\Phi_{kag}({z},{w}) =
{\frac{1}{4}
\left[ \begin {array}{cccccc} -2&0&2&0&0&0\\
\noalign{\medskip}0&0&1&-\sqrt{3}&-1&\sqrt {3}\\
\noalign{\medskip}-1&-\,\sqrt {3}&0&0&1&\sqrt{3}\\
\noalign{\medskip}2&0&-2\,{z}&0&0&0\\
\noalign{\medskip}0&0&-1&\sqrt {3}&\ol{z}{w}&-\sqrt{3}\ol{z}{w}\\
\noalign{\medskip}\ol{w}&\sqrt{3}\ol{w}&0&0&-1&-\sqrt {3}\end {array}
\right]}
\]

(ii)
The determinant of $\Phi_{kag}(z,w)$ is a multiple of
\[
zw(\ol{z}-1)(\ol{w}-1)(\ol{z}-\ol{w})
\]

(iii) The kagome framework is square-summably isostatic.
\end{prop}

We remark that the function matrix association above
has also arisen in engineering in the analysis of Hutchinson and Fleck \cite{hut-fle}
of the stresses and rigidity of  the kagome repetitive truss framework. This  is derived from the crystallographic perspective of  wave periodic flexes and Bloch's theorem.



\subsection{Symmetry equations for infinite frameworks.}
Let $\C$ be a crystal framework in $\bR^d$ with complex Hilbert space rigidity
operator $R$. Also let $\K_v^2, \K^2_e$ be as before, let $\K^2_{fl}$ be the space of square-summable infinitesimal flexes, let $\K^2_{str}$ be the space of square-summable infinitesimal stresses
and let $\M=  \K_v^2 \ominus \K^2_{fl}$, $\N=  \K_e^2 \ominus \K^2_{str}$
be their complementary spaces.

As well as commuting with the coordinate shift operators
the rigidity operator satisfies
commutation relations for every isometric symmetry of $\C$.
For example, let $S$ be an isometric (not necessarily linear) operator on $\bR^d$ which effects a symmetry of $\C$.
There is an induced unitary operator $S_e$ on the complex Hilbert
space $\ell^2(E)$ and an analogous operator $S_v$
on $\ell^2(V)$ which permute the standard basis elements. Write
$\tilde{S}_v$ for the isometric (not necessarily linear)
operator $ S_v \otimes S$  on $\H_v = \ell^2(V) \otimes \bC^d$
where $S$ is also viewed as an isometric operator on $\bC^d$.
Then we have the fundamental symmetry equation
$S_eR = R\tilde{S}_v$. (See \cite{owe-pow-2}.)

Moreover, if $\G$ is the symmetry group of $\C$ arising from
isometric transformations of $\bR^d$ and if  $\rho_e$ and $\tilde{\rho}_v$ are the associated
representations of $\G$ on $\K_e^2$ and $\K_v^2$ then we have the symmetry equations
\[
\rho_e(g)R = R\tilde{\rho}_v(g) \mbox{   for   } g \in \G.
\]
We have shown in \cite{owe-pow-2} how such symmetry equations may be used to
obtain a simple proof of a unitary equivalence which implies the Fowler-Guest formula \cite{fow-gue}
together with various generalisations. In the present setting we have
the following analogue which also leads to counting conditions for isostatic
and rigid frameworks.

\begin{thm}
Let $\C$ be a crystal framework with isometry
symmetry group $\G$ and let
\[
\tilde{\rho}_v = \rho_\M \oplus \rho_{fl}
\]
\[
\rho_e = \rho_\N \oplus \rho_{str}
\]
be the decompositions associated with the spaces of
square-summable infinitesimal flexes
and stresses. Then $\rho_\M$ and $\rho_\N $ are unitarily equivalent
representations. In particular if $\C$ is square summably isostatic
then $\tilde{\rho}_v $ and $\rho_e$ are unitarily equivalent.
\end{thm}

\begin{proof}
In the square-summably isostatic case  $R$ is a bounded operator
with trivial kernel and trivial cokernel. The partially isometric
part $U$ of the polar decomposition $R = U(R^*R)^{1/2}$ is therefore unitary and
it is a standard verification that this unitary also intertwines the
representations.
In general the symmetry equations show that
the space $\K^2_{fl}$ is reducing for $\tilde{\rho}_v $ and
$\K^2_{str}$ is reducing for $\rho_e$
 and so the asserted
direct sum decompositions do exist. Now the restriction of $R$ to $\M$
maps to $\N$ with trivial kernel and cokernel and so as before the restriction
representations are unitarily equivalent.
\end{proof}


\subsection{From matrix function to wave modes}
 For a crystal framework
$\C = (F_v, F_e, \bZ^d )$ with a given motif
let $\Phi: \bT^d \to M_{m,nd}$ be the associated   matricial symbol function.
\begin{definition}
The mode multiplicity function of $\C$ associated with the given motif and translation group is the function $\mu: \bT^d \to \bZ_+$
given by
$\mu(z) = \dim \ker \Phi (z).$
\end{definition}

For $d=2$ we also consider $\mu$ as being parametrised by coordinates
$s,t$ in $[0,1)\times [0,1)$ so that $(z,w)\in \bT^2$ corresponds to $(e^{2\pi is},
e^{2\pi it})$.
From the determinant calculations above we obtain the following.

\begin{prop}
(i) For the grid framework $\G_{\bZ^2}$ and the 8-edged motif the mode multiplicity
function has values $\mu(0,0)=2$,  $$\mu(s,0) = \mu(0,t) = 1,$$ if $s$ and
$t$ are nonzero, and is zero otherwise.

(ii) For the kagome framework $\G_{kag}$ and the 6-edged motif
the mode multiplicity
function has values  $\mu(0,0)= 2,$ at the origin while $\mu(s,s)=\mu(s,0) = \mu(0,t) = 1$ if $s\neq 0$ and
$t\neq $ and is zero otherwise.
\end{prop}

Consider now the \textit{wave flexes} of $\C$ which we define as the
infinitesimal flexes which are $1$-cell-periodic modulo a phase factor.

\begin{definition} Let $\C = (F_v, E_v, \bZ^d)$ be a crystal framework in $\bR^d$.
An (infinitesimal)  wave flex of $\C$ is a complex infinitesimal flex $u= (u_v)_{v\in V}$
which is wave periodic (or more precisely $1$-cell wave periodic) in the sense
that  for some vector $q$ in $\bR^d$
\[
u_{\kappa+n}=e^{2\pi i\langle q,n\rangle}u_\kappa
\]
for each vertex $\kappa$ in the motif set $F_v$ and each $d$-tuple
$n=(n_1,\dots ,n_d)$.
\end{definition}


The values of the mode multiplicity function corresponds to
the dimension of the spaces of wave flexes. To see this
suppose  that $\Phi(z)$ is the matricial symbol function
 for $\C =(G,p) = (F_v, F_e, \bZ^3)$, that $w \in \bT^3$ and that $\det \Phi(w)=0$.
Then $\Phi(w)u_m=0$ for some nonzero complex motif vector
$u_m$. For the Dirac delta function $\delta_w(z)$ on $\bT^3$ we have, informally,
$\Phi(z)F(z)=0$ for all $z$ on the $3$-torus where $F(z)$ is the function  $\delta_w(z)u_m$.
Thus, taking Fourier transforms it follows  that
the  wave periodic vector $u=\F(F)$ in $\K_v$
satisfies $R(G,p)u=0$. This shows that the bounded (and non square-summable)
vector
\[
u = (u_{\kappa,n})_{\kappa \in F_v,n\in\bZ^3} = \F(F)(\kappa,n)
= w^nu_\kappa
\]
is a wave  flex.
The Dirac delta argument can be rigourised in the usual manner and so we  obtain the following theorem.

\begin{thm} Let  $\C = (F_v, F_v, \bZ^d)$ be a crystal framework in $\bR^d$
with associated symbol function $\Phi(z)$ and suppose that $\C$ has no internal (finitely supported) self-stresses and no internal (finitely supported)
flexes. Then

(i) $d|F_v| = |F_e|$, the associated matricial symbol function
$\Phi(z)$ on the $d$-torus is square with full rank almost everywhere
and $\C$ is square-summably isostatic.

(ii) infinitesimal
wave flexes (that are 1-cell periodic-modulo-phase) exist, with phase factor
$w \in \bT^d$,
if and only if $~\det \Phi(w)=0$. In this case the dimension of the corresponding
space of wave flexes is $\dim \ker \Phi(w)$.
\end{thm}

It follows also from this that the zero set of the determinant of the matricial symbol function associated with the motif of an appropriate  supercell
coincides with the phases of supercell-periodic wave flexes.

Of particular computational and theoretical interest
is what one might refer to as the wave flex acquisition when a crystal framework deforms under a colossal flex to a framework with higher symmetry. This phenomenon serves as a model for the appearance of so-called Rigid Unit Modes (RUMs) in higher symmetry phases
of various material  crystals.

The following theorem generalises an interesting result
of Wegner \cite{weg} for tetrahedral crystals.
It may be viewed as an expression of the simple principle that additional
symmetry often entails additional flexibility.
Our proof applies to arbitrary crystal frameworks and is quite direct,
benefitting somewhat from the economy of operator theory formalism.
We remark that there are also  natural operator algebra perspectives that are
relevant to symmetry considerations. The result shows in particular for $d=3$ that the RUM set is typically
a union of surfaces, being the zero set of a single real-valued polynomial,
rather than the intersection of the zero sets of the real and imaginary
part of a complex polynomial.

For a given crystal framework and motif
write $\Omega$ for the subset of   $\bT^d$ formed by the  phases $\omega$ of the $1$-cell periodic wave flexes. We also refer to this informally as the RUM set or the RUMs of $\C$.

\begin{thm}
Let  $\C = (F_v, F_e, \bZ^d)$ be a crystal framework in $\bR^d$
with $d|F_v|=|F_e|$
and suppose that $\C$ possesses inversion symmetry.
Then the set $\Omega$
has the form
\[
\Omega = \bT^d \cap V(p)
\]
where $V(p)$ is the zero set of a complex polynomial $p(z_1,\dots ,z_d,\ol{z}_1 \dots , \ol{z}_d)$ which is real-valued on $\bT^d$.
\end{thm}

\begin{proof}
Let us denote the set of framework vertices and edges as
\[
\V = \{\kappa + n: \kappa \in F_v, n \in \bZ^d\}
\]
\[
\E = \{e + n: e \in F_v, n \in \bZ^d\}
\]
Effecting a translation, if necessary,
we may assume that the inversion is $\sigma : x \to -x$, that the unit cell has
inversion symmetry and $\sigma(F_v) = F_v$. It may or may not be possible to
 re-choose the noninternal edges
of $F_e$ so that that $\sigma(F_e) = F_e$.
Suppose first that this is the case. We show that $\det \Phi(z)$ itself is either real-valued or purely imaginary, from which the stated form for  $\Omega$ follows.

With the notation of Section 5.4 we have the symmetry equation
\[
\rho_e(\sigma)R = R\tilde{\rho}_v(\sigma).
\]
Recall that $\rho_e(\sigma)$ and $\tilde{\rho}_v(\sigma)$
are the isometric operators induced by the (isometric) symmetry element
$\sigma$. Taking Fourier transforms this equation takes the form
\[
U(\sigma)M_{\Phi(z)} = M_{\Phi(z)}V(\sigma)
\]
where $U(\sigma)$ and $V(\sigma)$ are the unitary operators determined by their
action on the distinguished orthonormal bases. In view of the assumption we have
\[
U(\sigma)(\eta_k \otimes z^l) = \eta_{\sigma_e(k)} \otimes z^{-l}, \quad  k = 1,\dots , |F_e|
\]
(for certain induced permutations $\sigma_e, \sigma_v$) and
\[
V(\sigma)(\xi_k^{x_i} \otimes z^l) = \xi_{\sigma_v(k)}^{x_i} \otimes z^{-l}, \quad  k = 1,\dots , |F_v|, i=1,\dots , d.
\]
That is, these operators have the form
\[
U(\sigma) = E_\sigma \otimes J, \quad  V(\sigma) = V_\sigma \otimes J
\]
where $J$ is the inversion unitary operator on $L^2(\bT^d)$ given by $(Jf)(z) = f(\ol{z})$,
and where $E_\sigma$ and $V_\sigma$ are scalar permutation matrices induced
by $\sigma$.
Substituting these forms
we see that
\[
\Phi(z) = (E_\sigma \otimes J)^{-1}\Phi(z)(V_\sigma \otimes J)
= E_\sigma^{-1}\Phi(\ol{z})V_\sigma .
\]
It follows that
\[
\det \Phi(z) = \det \Phi(\ol{z})\det(E_\sigma)^{-1}\det(V_\sigma) = (-1)^\tau\ol{\det(\Phi(z))}
\]
for some integer $\tau$, since the determinant polynomial has real coefficients and so the determinant is either real or purely imaginary.

In the general case because of edges lying on inversion axes (lines through the origin) $\sigma$ cannot be made to act freely on the motif edges. However
for each  motif edge index $k$ there is a monomial shift factor $z^{p(k)}$
such that with $D$ the diagonal matrix function
$$D = diag (z^{p(1)}, \dots , z^{p(d)})$$
we have
\[
U(\sigma) =D(E_\sigma \otimes J).
\]
Now we see that $\det \Phi(z) = (-1)^\tau z^p\ol{\det(\Phi(z))}$
for some multi-index $p$.
Write $\Phi = F_1 + iF_2$ with the $F_i$ real-valued and
$(-1)^\tau z^p= e^{i\gamma(z)}$ with $\gamma = \gamma(z)$ real.
Equating real and imaginary
parts it follows that either $\cos \gamma =1$ in which case
$\det \Phi(z) = \ol{\det(\Phi(z))}$ and $\Phi$ is real-valued,
or $F_1(z)=(\sin \gamma)(1-\cos \gamma)^{-1}F_2(z)$. In all cases,
for $z$ on the torus,
 $F_1(z)=0$ if and only if and $F_2(z)=0$, as required.
\end{proof}




\subsection{Honeycomb frameworks and the kagome net.}
 { The explicit formulation of the matricial symbol function makes clear a simple additive principle which is useful for calculation : adding internal edges to a unit cell results in
 adding the same number of
rows to the symbol function to form the new symbol function, and each of these
rows has the standard form with constant entries.

Consider the following honeycomb frameworks based on regular \\ hexagons.

\begin{center}
\begin{figure}[h]
\centering
\includegraphics[width=5cm]{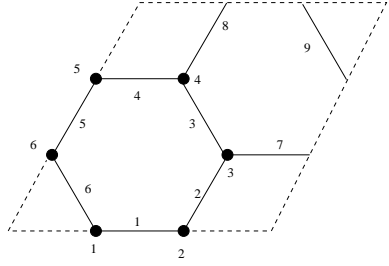}
\caption{A $3$-regular honeycomb framework.}
\end{figure}
\end{center}

\begin{center}
\begin{figure}[h]
\centering
\includegraphics[width=5cm]{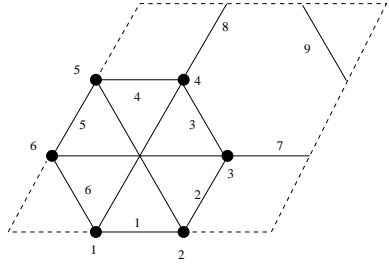}
\caption{A $4$-regular honeycomb framework.}
\end{figure}
\end{center}

The $3$-regular (symmetric) honeycomb framework has a $9$ by $12$ matricial symbol function. There are vectors in the kernel of the multiplication operator providing internal infinitesimal flexes and there are  evident bounded infinitesimal flexes associated with "parallel deformations", for example.

A $4$-regular framework may be constructed by
rigidifying the hexagon internal to the unit cell, by adding three cross edges.
The matricial symbol function is correspondingly enlarged by three
 scalar rows, appearing in rows $9$ to $12$ in the following matrix.
{\tiny
\[
 \left[ \begin {array}{cccccccccccc} -1/3&0&1/3&0&0&0&0&0&0&0&0&0
\\\noalign{\medskip}0&0&-1/6&-1/6\,\sqrt {3}&1/6&1/6\,\sqrt {3}&0&0&0&0
&0&0\\\noalign{\medskip}0&0&0&0&1/6&-1/6\,\sqrt {3}&-1/6&1/6\,\sqrt {3
}&0&0&0&0\\\noalign{\medskip}0&0&0&0&0&0&1/3&0&-1/3&0&0&0
\\\noalign{\medskip}0&0&0&0&0&0&0&0&1/6&1/6\,\sqrt {3}&-1/6&-1/6\,
\sqrt {3}\\\noalign{\medskip}1/6&-1/6\,\sqrt {3}&0&0&0&0&0&0&0&0&-1/6&
1/6\,\sqrt {3}\\\noalign{\medskip}0&0&0&0&-1/3&0&0&0&0&0&1/3\,\ol{z}&0
\\\noalign{\medskip}1/6\,\ol{w}&1/6\,\sqrt {3}\ol{w}&0&0&0&0&-1/6&-1/6\,\sqrt {3
}&0&0&0&0\\\noalign{\medskip}0&0&-1/6\,\ol{w}&1/6\,\sqrt {3}\ol{w}&0&0&0&0&1/6\,
\ol{z}&-1/6\,\sqrt {3}\ol{z}&0&0\\\noalign{\medskip}-1/3&-1/3\,\sqrt {3}&0&0&0&0
&1/3&1/3\,\sqrt {3}&0&0&0&0\\\noalign{\medskip}0&0&1/3&-1/3\,\sqrt {3}
&0&0&0&0&-1/3&1/3\,\sqrt {3}&0&0\\\noalign{\medskip}0&0&0&0&2/3&0&0&0&0
&0&-2/3&0\end {array} \right]
\]
}
The cross-barred hexagon is continuously rigid but is infinitesimally flexible. We refer to this flex as the in-out flex. (See Whiteley \cite{whi-1} for related discussions.) It is straightforward to show
that this  flex extends periodically  as a 1-cell-periodic infinitesimal flex of the framework.

The motif matrix function is obtained by evaluating at $z=w=1$. This has rank $8$
reflecting the four independent $1$-cell-periodic flexes corresponding to
horizontal translation, to vertical translation, to local rotation of the cross-barred hexagon, and to the periodic "in-out" flex.
}

We define the kagome net framework as the three dimensional crystal framework
$\G_{knet} = (F_v, F_e,\bZ^3)$
determined by the motif diagram in Figure 21.
This figure shows a tetrahedron internal to a parallelepiped unit cell
together with six additional non-internal edges, each of which extends a
tetrahedron edge as indicated.

\begin{center}
\begin{figure}[h]
\centering
\includegraphics[width=10cm]{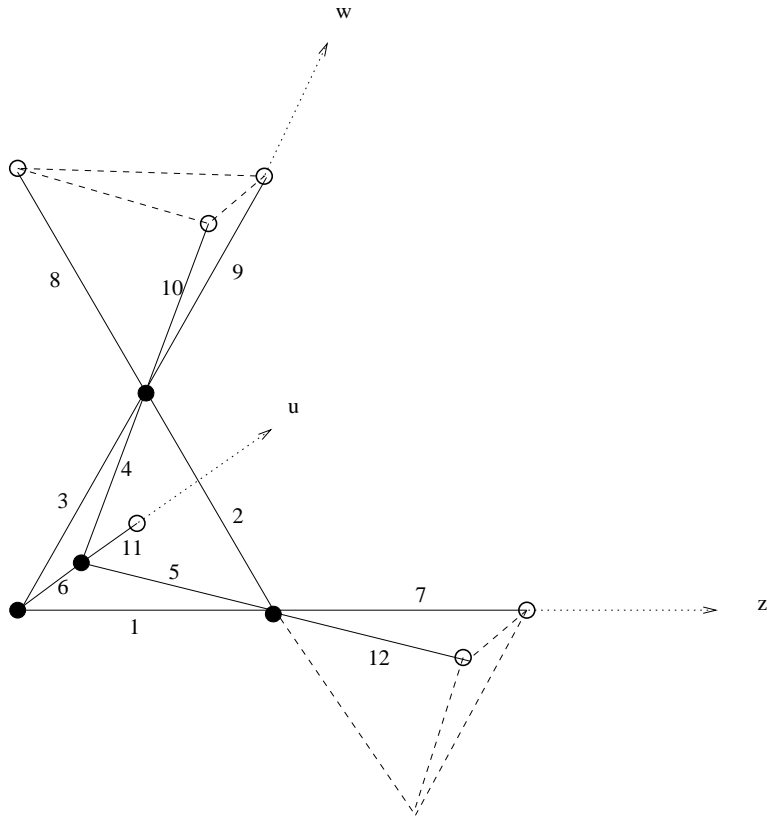}
\caption{A twelve-edged motif for the kagome net, $\G_{knet}$.}
\end{figure}
\end{center}

For the given motif the motif rigidity matrix $R_m(\G_{knet})$
takes the form of a $12$ by $12$ matrix valued function on the $3$-torus.
Different motifs, for the same translation group, give matrix functions
that are related by permutations on columns and monomial multiplication on rows.
It can be shown that the determinant of such a matrix function is equal to a
scalar and monomial multiple of the polynomial
\[
\left(z-1 \right)  \left(w-1 \right)\left( u-1 \right)
\left( z-w \right)    \left( w-u \right)  \left( u-
z \right).
\]
This is in complete analogy with the kagome framework and in the light of earlier discussions may draw the following conclusions.

\begin{thm}
(i) The kagome net $\G_{knet}$ is square-summably isostatic and
possesses  no internal infinitesimal flexes.

(ii) The mode multiplicity function $\mu : \bT^3 \to \bZ_+$ of the kagome net framework
has support equal to the intersection of $\bT^3$ with the six planes
\[
z=1, w=1, u=1, z=w, w=u,u=z.
\]
\end{thm}

\medskip


\subsection{Crystallography and Rigid Unit Modes.}
Rigidity theory and a Hilbert space operator viewpoint have led us
to determine, ab initio, the matricial symbol function $\Phi(z)$ of an
abstract crystal  framework and motif.
The zeros of $\det \Phi(z)$
correspond to the phases of infinitesimal periodic-modulo-phase wave flexes and the mode multiplicity function $\mu(w)$ of  $\Phi$ detects  the multiplicities of independent wave flexes for $w$.
There are close connections between these mathematical observations and the  Rigid Unit Modes (RUMs) that are observed in certain material crystals through diffraction experiments. This connection  is part of the motivation  for the formulation of the function $\mu(z)$.
 It seems to us that the explicit algorithm for the passage from crystal motif to matrix function will provide a useful computational
and theoretical tool for identifying  RUMs and their relationships.

Infinitesimal wave flexes  for abstract crystal frameworks  appear in the classical crystallography of Born von Karman theory. For example in the case of variants of
quartz the low energy excitation modes  result from rigid unit motions of $SiO_4$ molecular tetrahedra within a tetrahedral net. These modes are therefore modelled by the first order periodic-modulo-phase infinitesimal rigid motions
of abstract tetrahedral net frameworks.
For this connection see, for example, Chapter 3 of \cite{wil-pry}.
Comparisons of extensive simulations and experimental results  have shown that the RUM modes of simulated  crystals are closely correlated with observed RUMs. For this see the seminal paper of Giddy et al \cite{gid-et-al}, and also, for example, Dove et al \cite{dov-proc} and Goodwin et al \cite{goodwin2008}.

An interesting advance has been obtained recently by Wegner \cite{weg} who has derived mathematically, rather than through simulation, the RUM sets  for various idealised tetrahedral crystals that model the geometry of material crystals, including $\beta$-cristobalite, HP tridymite, $\beta$-quartz, $\alpha$-cristobalite and $\alpha$-quartz.  These models are standard framework models that coincide with crystal bar-joint frameworks in our terminology and the results are obtained by determining vanishing determinants (zero sets of $\det \Phi$ in our formalism) by means of computer algebra and symmetry reductions.  In particular Wegner obtains an analytic derivation of some of the curious  surfaces observed in experiments of Dove et al.





\bibliographystyle{abbrv}
\def\lfhook#1{\setbox0=\hbox{#1}{\ooalign{\hidewidth
  \lower1.5ex\hbox{'}\hidewidth\crcr\unhbox0}}}

\end{document}